\documentclass[a4paper,11pt,dvipdfm]{amsart}
\usepackage{amsmath}
\usepackage{amssymb}
\usepackage{enumerate}
\usepackage{amscd}
\usepackage{graphics}
\usepackage{latexsym}
\usepackage{verbatim} 
\usepackage{url}
\usepackage{bbm}

\theoremstyle{plain}
\newtheorem{thm}{{\bf Theorem}}[section]
\newtheorem{cor}[thm]{{\bf  Corollary}}
\newtheorem{prop}[thm]{{\bf Proposition}}
\newtheorem{lemma}[thm]{{\bf Lemma}}

\newtheorem{claim}[thm]{{\bf Claim}}

\theoremstyle{definition}
\newtheorem{define}[thm]{{\bf Definition}}
\newtheorem{note}[thm]{{\bf Note}}

\newtheorem{question}[thm]{{\bf Question}}

\newcommand{\cf}{\mathord{\mathrm{cf}}}

\newcommand{\size}[1]{\left\vert {#1} \right\vert}

\newcommand{\seq}[1]{\langle {#1} \rangle}

\newcommand{\ka}{\kappa}
\newcommand{\la}{\lambda}
\newcommand{\om}{\omega}

\newcommand{\bbP}{\mathbb{P}}
\newcommand{\bbQ}{\mathbb{Q}}

\newcommand{\calB}{\mathcal{B}}

\newcommand{\calF}{\mathcal{F}}
\newcommand{\calG}{\mathcal{G}}

\newcommand{\calN}{\mathcal{N}}

\newcommand{\calU}{\mathcal{U}}
\newcommand{\calV}{\mathcal{V}}

\newcommand{\Add}{\mathrm{Add}}

\newcommand{\fraku}{\mathfrak{u}}
\newcommand{\ZFC}{\mathsf{ZFC}}

\title[Monotonicity of the ultrafilter number function]{Monotonicity of the ultrafilter number function}
\author[T. Usuba]{Toshimichi Usuba}
\address[T. Usuba]
{Faculty of Fundamental Science and Engineering,
Waseda University, 
Okubo 3-4-1, Shinjyuku, Tokyo, 169-8555 Japan}
\email{usuba@waseda.jp}

\begin{document}
\begin{abstract}
We investigate whether the ultrafilter number function $\ka \mapsto \mathfrak{u}(\ka)$ on the cardinals is monotone,
that is, whether $\fraku(\la) \le \fraku(\ka)$  holds for all cardinals $\la < \ka$ or not.
We show that monotonicity can fail, but the failure has large cardinal strength.
On the other hand, we prove that there are many restrictions of the failure of monotonicity.
For instance, if $\ka$ is a  singular cardinal with
countable cofinality or a strong limit singular cardinal, then $\fraku(\ka) \le \fraku(\ka^+)$ holds.
 
\end{abstract}

\maketitle
\section{Introduction}
For a cardinal $\ka$,
let $\fraku(\ka)$ denote the ultrafilter number at $\ka$,
that is, the minimal cardinality of a base of a uniform ultrafilter over $\ka$.
A central research on the ultrafilter number 
is the comparison 
between $\mathfrak{u}(\ka)$ and other cardinal invariants at $\ka$:
A typical one is, whether $\fraku(\ka)<2^\ka$ is consistent or not.
In this context,
the ultrafilter number at $\om$ is studied extensively, 
and recently the ultrafilter number at
% large cardinals and 
singular cardinals  
is actively researched (e.g., \cite{BJ1},  \cite{GS}, \cite{GS2}, \cite{G1}, \cite{G2}, \cite{HS2}).
In contrast, little is  known about the ultrafilter number at successor cardinals.
While it is still open if $\fraku (\om_1)<2^{\om_1}$ is consistent or not,
Raghavan and Shelah \cite{RS} proved the consistency
of $\fraku(\om_{\om+1})<2^{\om_{\om+1}}$ 
under some large cardinal assumption.
However it is unknown if this large cardinal assumption is necessary or not.

This paper also investigates the ultrafilter number, but
does not aim to study the comparison between $\fraku (\ka)$ and other cardinal invariants at $\ka$.
Instead, we identify the ultrafilter numbers as the 
function $\ka \mapsto \fraku(\ka)$ on the cardinals, 
and we shall consider the behavior of the ultrafilter number function.
Under GCH, the ultrafilter number $\fraku(\ka)$ is just $2^\ka=\ka^+$,
hence the ultrafilter number function is
just the continuum function $\ka \mapsto 2^\ka$, 
in particular it is strictly increasing.
On the other hand, it is easy to construct a model 
in which the ultrafilter number function is not strictly increasing:
Assuming GCH, just add $\om_2$ many Cohen real. Then $\fraku (\om)=\fraku(\om_1)=\om_2$
holds in the generic extension
(e.g., see Lemma \ref{10}).
Under this viewpoint, a natural question arises.
This is a main topic of this paper. 
\begin{question}
Is the ultrafilter number function \emph{monotone}?
That is, for every cardinal $\la$ and $\ka$ with
$\la \le \ka$, does $\fraku(\la) \le \fraku(\ka)$ always hold?
\end{question}
Hart and van Mill also asked a similar question (Question 63 in \cite{HM}) in
the context of set-theoretic topology.

For some cardinal invariant, it is easy to show that
monotonicity can fail:
For the dominating number $\mathfrak{d}(\ka)$, 
suppose GCH and add $\om_3$ many Cohen reals.
Then $\mathfrak{d}(\om)=\om_3>\om_2=\mathfrak{d}(\om_1)$ holds in the generic extension.
However, this naive approach does not work for the ultrafilter numbers;
In that generic extension, we have $\mathfrak{u}(\om)=\mathfrak{u}(\om_1)=\om_3$
(see Lemma \ref{10} again).

In this paper, we prove that the failure of monotonicity is actually consistent.
On the other hand, we also prove that the failure of it has  large cardinal strength.
Moreover there are many $\ZFC$-restrictions of the failure of monotonicity.
First, 
by using Raghavan and Shelah's argument, 
we show that monotonicity of the ultrafilter number function can fail.
\begin{thm}\label{main1}
Relative to a certain large cardinals, the following are consistent\footnote{For the exact large cardinal
assumption in each item, see the referenced number next to it.}:
\begin{enumerate}
\item There are cardinals $\ka, \la$
such that $\la<\ka$ but $\fraku(\ka)<\fraku(\la)$ holds (Theorem \ref{4.4})
\item $\fraku(\om_{\om+1})<\fraku(\om_1)$ holds (Theorem \ref{3.10}).
\item There is a singular cardinal $\ka$ with cofinality $\om_1$
such that $\mathfrak{u}(\ka^+)<\mathfrak{u}(\ka)$ holds (Theorem \ref{3.11}).
\item There is a regular cardinal $\ka$ 
such that $\mathfrak{u}(\ka^+)<\mathfrak{u}(\om_1)$ holds (Theorem \ref{3.11}).
\item There is a cardinal $\ka$
such that $\om_\om<\ka$ but $\fraku(\ka)<\fraku(\om_\om)$ holds (Theorem \ref{7.5}).
\item There is a cardinal $\ka$ such that
$\mathfrak{u}(\ka^{+\om_1})<\mathfrak{u}(\om_\om)$ holds (Proposition \ref{7.19}).
\end{enumerate}

\end{thm}
Raghavan and Shelah's argument involves indecomposable ultrafilters (see below for indecomposable ultrafilters),
but it is not clear if the use of indecomposable ultrafilters is necessary 
for  their theorem. 
However, as we will see the failure of monotonicity is directly connected with
indecomposability. Using this connection, 
we also prove that there are many restrictions of the failure of monotonicity.
\begin{thm}
\begin{enumerate}
\item $\fraku(\om) \le \fraku(\ka)$ holds for every cardinal $\ka$ (Proposition \ref{5.3+})
\item For cardinals $\ka,\la$, if $\la$ is regular, $\la<\ka$, but $\fraku(\ka)<\fraku(\la)$,
then $\la^{+\om} \le \ka$ (Proposition \ref{21}).
\item If $\ka$ is a singular cardinal with countable cofinality, or
a strong limit singular cardinal, then $\mathfrak{u}(\ka) \le \mathfrak{u}(\ka^+)$ (Corollary \ref{29},
Theorem \ref{7.9}). 
\item For a singular cardinal $\ka$, if $\mathfrak{u}(\ka^+) <\fraku(\ka)$,
then the set $\{\la<\ka \mid$ $\la$ is regular, $\fraku(\la)>\fraku(\ka^+)\}$ is bounded in $\ka$
(Theorem \ref{5.15}).
\item If $\square(\ka)$ holds for a regular $\ka$,
then $\mathfrak{u}(\la) \le \mathfrak{u}(\ka)$ holds for every $\la<\ka$ (Corollary \ref{8.13}).
\end{enumerate}
\end{thm}

To prove (3) and (4) in this theorem, we  will show some new results about
(in)decomposable ultrafilters, which are interesting in its own right.
Recall that, for a cardinal $\ka$ and an ultrafilter $\calU$ over a set $S$,
$\calU$ is \emph{$\ka$-decomposable} if
there is a function $f:S \to \ka$ such that
$f^{-1}(X) \notin \calU$ for every $X \in [S]^{<\ka}$,
and it is \emph{$\ka$-indecomposable} if it is not $\ka$-decomposable. 
\begin{thm}[Theorem \ref{27}]\label{thm1.4}
Let $\ka$ be a singular cardinal,
and $\calU$ an ultrafilter over a set $S$.
If $\calU$ is $\ka^+$-decomposable and $\cf(\ka)$-decomposable,
then $\calU$ is $\ka$-decomposable as well.
\end{thm}
Since every $\ka$-decomposable ultrafilter is $\cf(\ka)$-decomposable,
this theorem tell us that
$\cf(\ka)$-decomposability and $\ka$-decomposability are equivalent on
$\ka^+$-decomposable ultrafilters.

For a limit cardinal $\ka$, let us say that an ultrafilter $\calU$ is \emph{almost $\mathop{<}\ka$-decomposable}
if there is $\la<\ka$ such that $\calU$ is $\mu$-decomposable for every regular $\mu$ with $\la<\mu<\ka$.
Kunen and Prikry (\cite{KP}) showed
that, if $\ka$ is singular, $\calU$ is a $\ka^+$-decomposable but $\cf(\ka)$-indecomposable ultrafilter,
then $\calU$ is almost $\mathop{<}\ka$-decomposable.
Lipparini (\cite{L}) showed the converse for $\cf(\ka)$-indecomposable ultrafilters:
If $\calU$ is $\cf(\ka)$-indecomposable,
then $\calU$ is $\ka^+$-decomposable if and only if
$\calU$ is almost $\mathop{<}\ka$-decomposable
(see Theorem \ref{KPth} and \ref{Li} below for details).
By using Theorem \ref{thm1.4}, we shall prove that $\cf(\ka)$-indecomposable assumption
in their theorem can be replaced by $\ka$-indecomposability
(Proposition \ref{37}, Corollary \ref{7.13}).

Finally we show that the failure of monotonicity has large cardinal strength,
so the large cardinal assumption in Theorem \ref{main1} cannot be eliminated.
\begin{thm}[Theorem \ref{8.4}]
The following theories are equiconsistent:
\begin{enumerate}
\item $\ZFC$+``there is a measurable cardinal''.
\item $\ZFC$+``there are cardinals $\ka, \la$
such that $\la<\ka$ but $\fraku(\ka)<\fraku(\la)$''.
\end{enumerate}
\end{thm}

\begin{thm}[Theorem \ref{8.16}]
Suppose there are cardinals $\ka, \la$ such that
$\la \le \ka$ but $\fraku(\ka^+)<\fraku(\la)$.
Then there is an inner model with a proper class of strong cardinals.
\end{thm}

\section{Preliminaries}
Throughout  this paper,
a \emph{filter} means a proper filter over a non-empty set.
A filter $\calF$ over a set $S$ is \emph{uniform}
if $\size{X}=\size{S}$ for every $X \in \calF$.
For an ultrafilter $\calU$ over a set $S$,
a subfamily $\calG \subseteq \calU$ is a \emph{base} for $\calU$
if for every $Y \in \calU$ there is $X \in \calG$ with $X \subseteq Y$.
Let $\chi(\calU)=\min \{\size{\calG} \mid \calG \subseteq \calU$ is a base$\}$.
$\chi(\calU)$ is called the \emph{character} of $\calU$.
Let $\mathfrak{u}(\ka)$ be the cardinal $\min\{\chi(\calU) \mid \calU$ is a uniform ultrafilter over $\ka\}$,
it is called the \emph{ultrafilter number at $\ka$}.
It is clear that $\mathfrak{u}(\ka) \le 2^\ka$.
\begin{thm}[Brendle and Shelah, Proposition 1.4 in \cite{BS}]%[Garti and Shelah \cite{GS}]\label{fact2.1}
$\mathfrak{u}(\ka) \ge \ka^+$ and $\cf(\mathfrak{u}(\ka))>\om$.
\end{thm}

Let  $\la$ be a cardinal.
A filter $\calF$ over a set $S$ is said to be \emph{$\la$-decomposable}
if there is a family $\{A_\alpha \mid \alpha<\la\}$
such that $\bigcup_{\alpha<\la} A_\alpha =S$
and $\bigcup_{\alpha \in X} A_\alpha \notin \calF$
for every $X \in [\la]^{<\la}$.
$\calF$ is \emph{$\la$-indecomposable} if
it is not $\la$-decomposable.
The following lemmas can be verified easily:
\begin{lemma}
For a filter $\calF$ over a set $S$ and a cardinal $\la$,
the following are equivalent:
\begin{enumerate}
\item $\calF$ is $\la$-decomposable.
\item There is $f:S \to \la$ such that
for every $X \subseteq \la$ with size $<\la$,
we have $f^{-1}(X) \notin \calF$.
\end{enumerate}
\end{lemma}
\begin{lemma}\label{2.3}
If $\calF$ is a $\la$-indecomposable filter, then every filter extending $\calF$ is
$\la$-indecomposable.
\end{lemma}

It is known that the existence of an indecomposable uniform ultrafilter is 
a large cardinal property.

\begin{thm}[Prikry and Silver \cite{PS}]\label{PS theorem}
Let $\la<\ka$ be regular cardinals.
If $\ka$ has a $\la$-indecomposable uniform ultrafilter,
then every stationary subset of $\{\alpha<\ka \mid \cf(\alpha)=\la\}$ reflects,
that is, for every stationary $S \subseteq \{\alpha<\ka \mid \cf(\alpha)=\la\}$,
there is $\beta<\ka$ such that $\cf(\beta)>\om$ and $S \cap \beta$ is stationary in $\beta$.
\end{thm}
It is known that such a stationary reflection property has large cardinal strength,
at least a Mahlo cardinal. See also Inamdar and Rinot \cite{IR}, and
Lambie-Hanson and Rinot \cite{LHR} for recent developments 
in this direction.

If $\ka$ is regular, then the set $\{\alpha<\ka^+ \mid \cf(\alpha)=\ka\}$
is a non-reflecting stationary set.
By this observation, we have:
\begin{lemma}\label{2.4}
If $\ka$ is regular,
then every uniform ultrafilter over $\ka^+$ is $\ka$-decomposable.
\end{lemma}
We note that this observation is immediate  by
Kunen and Prikry's result (\cite{KP}) that if $\ka$ is regular,
then every $\ka^+$-descendingly incomplete ultrafilter is
$\ka$-descendingly incomplete.

If $\delta$ is a cardinal,
then every $\delta^+$-complete ultrafilter is $\delta$-indecomposable,
and an ultrafilter is $\sigma$-complete if and only if it is $\omega$-indecomposable.

Non-large cardinal can have an indecomposable ultrafilter.
\begin{thm}[Ben-David and Magidor \cite{BM}]\label{BM theorem}
Suppose GCH and there is a supercompact cardinal.
Then there is a forcing extension in which
GCH holds and $\om_{\om+1}$ has a uniform ultrafilter $\calU$
such that $\calU$ is $\om_n$-indecomposable for every $0<n<\om$.
\end{thm}

Let $\mu$, $\nu$ be cardinals.
An ultrafilter $\calU$ over a set $S$ is \emph{$(\mu, \nu)$-regular}
if there is a family $\calF \subseteq \calU$ of size $\nu$
such that  for every $\calF' \in [\calF]^\mu$,
we have $\bigcap \calF'=\emptyset$.
It is easy to check that an ultrafilter $\calU$ over $S$ is $(\mu, \nu)$-regular if
and only if there is a family $\{x_s \mid s \in S\}$ such that
$x_s \in [\nu]^{<\mu}$, and for every $i<\nu$, the set $\{s \in S \mid i \in x_s\}$ is in $\calU$.

The following lemma is folklore,
and we give a proof of it for the completeness.
\begin{lemma}\label{reg-indecomp}
Every $\nu$-indecomposable ultrafilter is not
$(\om, \nu)$-regular.
\end{lemma}
\begin{proof}
For a given $\{x_s \mid s \in S\}$ with $x_s \in [\nu]^{<\om}$,
suppose $\{s \in S \mid \beta \in x_s\} \in \calU$ for every $\beta <\nu$.
Because $\nu^{<\om}=\nu$ and $\calU$ is $\nu$-indecomposable,
we can find $A \subseteq [\nu]^{<\om}$ such that
$\size{A}<\nu$ and $\{s \in S \mid x_s \in A \} \in \calU$.
We have $\size{\bigcup A}<\nu$,
hence there is $\beta \in \nu \setminus \bigcup A$.
But then $\{s \in S \mid \beta \notin x_s \} \in \calU$,
this is a contradiction.
\end{proof}

Let $\calU$ be an ultrafilter over a set $S$.
For a set  $T$ and $f:S \to T$,
let $f_*(\calU)=\{X \subseteq T \mid f^{-1}(X) \in \calU\}$.
$f_*(\calU)$ is an ultrafilter (possibly principal) over $T$.
The next lemma is folklore and can be verified easily.
\begin{lemma}\label{2.8}
Let $\calU$ be an ultrafilter over a set $S$,
and $\la$ a cardinal.
\begin{enumerate}
\item Let $f:S \to T$.
Then $\chi(f_*(\calU)) \le \chi(\calU)$, and 
if $\calU$ is $\la$-indecomposable, then so is
$f_*(\calU)$.
\item $\calU$ is
$\la$-decomposable if and only if
there is a function $g:S \to \la$
such that $g_*(\calU)$ is a uniform ultrafilter over $\la$.
\end{enumerate}
\end{lemma}

For a regular cardinal $\ka$ and a non-empty set $X$,
let $\mathrm{Add}(\ka, X)$ be the poset of
all partial functions $p:\ka \times X \to 2$ with size $<\ka$.
The order is given by the reverse inclusion.
The poset $\mathrm{Add}(\ka, X)$ is  $\ka$-closed, satisfies the $(2^{<\ka})^+$-c.c.,
and adds $X$-many new subsets of $\ka$.
For $Y \subseteq X$, we will identify $\Add(\ka, Y)$ as a complete suborder of $\Add(\ka, X)$.

\section{Raghavan and Shelah's theorem}
In this section we give a slightly general form of Raghavan and Shelah's theorem with relatively simple proof.
Using this theorem, we will show the consistency of the failure of monotonicity of the ultrafilter
number function in the next section.

\begin{thm}[Raghavan and Shelah \cite{RS}]\label{3}
Let $\ka$ and $\mu$ be uncountable cardinals such that:
\begin{enumerate}
\item $\cf(\mu)< \ka<\mu$.
\item $\nu^\ka <\mu$ for every $\nu<\mu$.
\item $\ka$ has a $\cf(\mu)$-indecomposable uniform ultrafilter $\calU$.
\end{enumerate}
Let $\bbP$ be a poset such that $\bbP$ has the $\cf(\mu)$-c.c.\,and $\size{\bbP}\le \mu$.
Then $\bbP$ forces that $\fraku(\ka) \le \mu$.
\end{thm}
\begin{proof}
We may assume that $\size{\bbP}=\mu$, otherwise the assertion is trivial.
Fix an enumeration $\{p_i \mid i<\mu\}$ of $\bbP$.
For $\gamma<\mu$, let us say that a $\bbP$-name $\dot A$ is 
a \emph{$\gamma$-nice name}
if $\dot A$ is of the form $\bigcup_{\alpha<\ka} (\{{\check \alpha}\} \times I_\alpha)$,
where each $I_\alpha$ is an antichain in $\bbP$ and $I_\alpha \subseteq \{p_i \mid i<\gamma\}$
($I_\alpha=\emptyset$ is possible).

\begin{claim}
For every $\bbP$-name $\dot B$ with $\Vdash \dot B \subseteq \check \ka$,
there is $D \in \calU$, $\gamma<\mu$, and
a $\gamma$-nice name $\dot A$ such that 
$\Vdash \dot B \cap \check D=\dot A$.
\end{claim}
\begin{proof}
For each $\alpha<\ka$,
take a maximal antichain $J_\alpha$ in $\bbP$ such that
each $p \in J_\alpha$ decides the statement $\check \alpha \in \dot B$.
Since $\bbP$ has the $\cf(\mu)$-c.c.,
there is $\gamma_\alpha<\mu$ such that $J_\alpha \subseteq \{p_i \mid i<\gamma_\alpha\}$.
$\mathcal{U}$ is $\cf(\mu)$-indecomposable, 
hence we can find $\gamma<\mu$ such that
$D=\{\alpha<\ka \mid \gamma_\alpha \le \gamma\} \in \calU$.
Let $\dot A=\{ \seq{\check \alpha, p} \mid \alpha \in D, p \in J_\alpha, p \Vdash \check \alpha \in \dot B\}$.
Clearly $\dot A$ is a $\gamma$-nice name and $\Vdash \dot A \subseteq \check D$.
We shall show that $\Vdash \dot B \cap \check D =\dot A$.

Fix $q \in \bbP$ and $\alpha<\ka$, and first suppose $q \Vdash \check \alpha \in \dot A$. 
We have $\alpha \in D$ by the definition of $\dot A$.
Pick $p \in J_\alpha$ which is compatible with $q$.
Take $r \le p,q$.
Since $r \Vdash \check \alpha \in \dot A$,
there is some $\seq{\check \alpha, p'} \in \dot A$
such that $p'$ is compatible with $r$.
Since $p,p' \in J_\alpha$ and $J_\alpha$ is an antichain, we have $p=p'$.
Hence $p \Vdash \check \alpha \in \dot B$,
and $r \Vdash \alpha \in \dot B \cap \check D$.

For the converse, suppose $q \Vdash \check \alpha \in \dot B \cap \check D$.
Again, pick $p \in J_\alpha$ which is compatible with $q$. Take $r \le p, q$.
We have $p \Vdash \check \alpha \in \dot B$,
hence $\seq{\check \alpha,p} \in \dot A$,
and $p \Vdash \check \alpha \in \dot A$.
Thus  $r \Vdash \check \alpha \in \dot A$.
\end{proof}

Let $\calN=\{\dot A \mid \dot A$ is a $\gamma$-nice name for some $\gamma<\mu\}$.
Since $\bbP$ has the $\cf(\mu)$-c.c., $\size{\{p_i \mid i<\gamma\}}<\mu$ for $\gamma<\mu$,
and $\nu^\ka<\mu$ for every $\nu<\mu$,
we have $\size{\calN} \le \mu$.

Take a $(V, \bbP)$-generic $G$. In $V[G]$, take a uniform ultrafilter $\calV$ over $\ka$ extending $\calU$.
Let $\calG=\{\dot A_G \in \calV \mid \dot A \in \calN\}$.
We have $\size{\calG} \le \mu$ and $\calG\subseteq \calV$.
We prove that for every $B \in \calV$, there is some $A \in \calG$ with $A \subseteq B$.
This completes our proof.

Fix $B \in \calV$, and take a $\bbP$-name $\dot B$ for $B$.
We may assume that $\Vdash \dot B \subseteq \check \ka$ in $V$.
By the claim, there is $\dot A \in \calN$ and $D \in \calU$ such that
$\Vdash \dot B \cap \check D=\dot A$.
Since $\calU \subseteq \calV$, we have $\dot A_G=B \cap D \in \calV$ and $\dot A_G \in \calG$.
\end{proof}

\begin{thm}[\cite{RS}]
Let $\ka$ be a measurable cardinal,
and $\mu>\ka$ a strong limit singular cardinal with $\om_1 \le \cf(\mu) <\ka$.
Then $\mathrm{Add}(\om, \mu)$ forces that
$\mathfrak{u}(\ka)<2^\ka$.
\end{thm}
\begin{proof}
Note that $\size{\Add(\om, \mu)}=\mu$ and 
every $\ka$-complete ultrafilter over $\ka$ is $\cf(\mu)$-indecomposable.
Hence, in the generic extension,
we have $\mathfrak{u}(\ka)\le \mu=2^\om \le 2^\ka$
by Theorem \ref{3},
and we know $\mu<2^\ka$ since $\cf(\mu)<\ka$.
\end{proof}

\begin{thm}[\cite{RS}]
Suppose the continuum hypothesis. Let $\ka$ be a measurable cardinal,
and $\mu>\ka$ a strong limit singular cardinal with $\om_2 \le \cf(\mu) <\ka$.
Then $\mathrm{Add}(\om, \ka) \times \mathrm{Add}(\om_1, \mu)$
forces that $\ka=2^\om$ and $\mathfrak{u}(\ka)<2^{\ka}$.
\end{thm}
\begin{proof}
We know that the cardinality of $\mathrm{Add}(\om, \ka) \times \mathrm{Add}(\om_1, \mu)$ is just $\mu$, 
and has the $\om_2$-c.c.
By standard arguments, $\mathrm{Add}(\om, \ka) \times \mathrm{Add}(\om_1, \mu)$
forces that $\ka=2^\om$ and $2^\ka > 2^{\om_1} = \mu$.
By Theorem \ref{3}, we have $\mathfrak{u}(\ka) \le \mu<2^\ka$ in the generic extension.
\end{proof}

The same argument together with Theorem \ref{BM theorem} proves:
\begin{thm}[\cite{RS}]
Suppose there is a supercompact cardinal.
Then there is a forcing extension in which $\mathfrak{u}(\om_{\om+1})<2^{\om_{\om+1}}$ holds.
\end{thm}

\begin{thm}
Let $\ka$ be a measurable cardinal,
and $\mu>\ka$ a strong limit singular cardinal with $\om_1 \le \cf(\mu) <\ka$.
Let $\calU$ be a normal measure over $\ka$, and $\bbP_\calU$ the Prikry forcing associated with $\calU$.
Then $\bbP_\calU \times \mathrm{Add}(\om, \mu)$ forces that
$\cf(\ka)=\om$ and $\mathfrak{u}(\ka)<2^\ka$.
\end{thm}
\begin{proof}
In the Prikry forcing extension $V^{\bbP_\calU}$, it is known that $\calU$ generates a
$\nu$-indecomposable uniform filter for every regular uncountable $\nu<\ka$ 
(Prikry, e.g., see Exercise 18.6 in Kanamori \cite{K}). By Lemma \ref{2.3}, every ultrafilter
extending $\calU$ is $\nu$-indecomposable for regular uncountable $\nu<\ka$.
Hence, in $V^{\bbP_\calU}$,
$\ka$ is a singular cardinal with countable cofinality and
has a $\cf(\mu)$-indecomposable uniform ultrafilter.
Also note that $\mu$ remains strong limit in $V^{\bbP_\calU}$.
By Theorem \ref{3}, 
forcing with $\Add(\om, \mu)$ over $V^{\bbP_\calU}$ forces
that $\ka$ is singular with countable cofinality and 
$\mathfrak{u}(\ka) \le \mu$.
Moreover $\mu<2^\ka$ since $\cf(\mu)<\kappa$.
\end{proof}

\section{Models with the failure of monotonicity 1}\label{sec1}
In this section we  construct models in which monotonicity of the ultrafilter number function fails.
For this sake, we need the following lemmas.

The next lemma is folklore.
\begin{lemma}\label{10}
Let $\ka$ be a cardinal,
$\mu>\ka$ a cardinal, and $\delta \le  \ka$ a regular cardinal.
If $\mathrm{Add}(\delta, \mu)$ has the $\ka'$-c.c.
for some $\ka'<\mu$ and preserves $\ka$,
then it forces $\mathfrak{u}(\ka) \ge \mu$.
\end{lemma}
\begin{proof}
Suppose not.
Take an $\mathrm{Add}(\delta, \mu)$-name $\dot \calU$ for a uniform 
ultrafilter over $\ka$ with character $\nu<\mu$,
and $\mathrm{Add}(\delta, \mu)$-names $\{\dot A_\alpha \mid \alpha<\nu\}$
of a base for $\calU$.
Fix a sufficiently large regular cardinal $\theta$,
and take $M  \prec H_\theta$ such that $\nu+\ka \subseteq M$, $\{\dot A_\alpha \mid \alpha<\nu\} \in M$, and 
$\size{M}<\mu$.
By the $\ka'$-c.c. of $\mathrm{Add}(\delta, \mu)$,
we may assume that for every $\alpha <\nu$,
$\dot A_\alpha$ is an $\mathrm{Add}(\delta, M \cap \mu)$-name.
Take a $(V, \mathrm{Add}(\delta, \mu))$-generic $G$,
and let $H=G \cap \mathrm{Add}(\delta, M \cap \mu)$,
which is $(V, \mathrm{Add}(\delta, M \cap \mu))$-generic.
We know $(\dot A_\alpha)_G=(\dot A_\alpha)_H$,
hence $(\dot A_\alpha)_G \in V[H]$.
$\Add(\delta, \mu \setminus M)$ is the quotient forcing from
$V[H]$ to a generic extension of $\Add(\delta, \mu)$.
Hence, by the standard genericity argument,
forcing $\mathrm{Add}(\delta, \mu \setminus M)$ over $V[H]$ adds a 
new set $B \in [\ka]^\ka$ such that
$(\dot A_\alpha)_G \nsubseteq B$ and
$(\dot A_\alpha)_G \nsubseteq \ka \setminus B$ for every $\alpha<\nu$,
this contradicts to that $\{(\dot A_\alpha)_G \mid \alpha<\nu\}$ is a base for $\calU$.
\end{proof}

\begin{lemma}\label{11}
Let $\ka$ be a cardinal,
$\mu>\ka$ a cardinal with $\cf(\mu)=\cf(\ka)$, and $\delta < \ka$ a regular cardinal.
If $\mathrm{Add}(\delta, \mu)$ satisfies the $\ka'$-c.c. for some $\ka'<\mu$,
preserves $\ka$, and $\size{\mathrm{Add}(\delta, \mu)} =\mu$,
then 
it forces $\mathfrak{u}(\ka) \neq \mu$.
\end{lemma}
\begin{proof}
Suppose not. Take an $\mathrm{Add}(\delta, \mu)$-name $\dot \calU$ for a uniform 
ultrafilter over $\ka$ with character $\mu$,
and $\mathrm{Add}(\delta, \mu)$-names $\{\dot A_\alpha \mid \alpha<\mu\}$
of a base for $\calU$.
We may assume that each $\dot A_\alpha$ is a nice name,
that is, there is a family $\{J_k \mid k<\ka\}$ of antichains in $\mathrm{Add}(\delta, \mu)$
such that $\dot A_\alpha=\bigcup_{k<\ka}( \{\check k\} \times J_k)$

Fix a sufficiently large regular cardinal $\theta$.
Since $\cf(\mu)=\cf(\ka)$,
we can take a sequence $\seq{M_i \mid i<\ka}$
such that:
\begin{enumerate}
\item $M_i \prec H_\theta$ and $M_i$ contains all relevant objects.
\item $ M_i \subseteq M_j$ for $i<j<\ka$.
\item $\ka+\ka' \subseteq M_i$.
\item $\size{ M_i} <\mu$.
\item $\mu \subseteq \bigcup_{i<\ka} M_i$.
\end{enumerate}
For $i<\ka$,
since $\mathrm{Add}(\delta, \mu)$ has the $\ka'$-c.c.,
we may assume that for every $\alpha \in M_i \cap \mu$,
$\dot A_\alpha$ is an $\mathrm{Add}(\delta, M_i \cap \mu)$-name;
$\dot A_\alpha$ is of the form $\bigcup_{k<\ka}( \{\check \alpha\} \times J_k)$
for some antichains $\{J_k \mid k<\ka\} \in M_\alpha$.
Since $\ka, \ka' \subseteq  M_\alpha$ and
$\size{J_k}<\ka'$, we have $\dot A_\alpha \subseteq M_\alpha$,
hence it can be seen as an $\mathrm{Add}(\delta, M_i \cap \mu)$-name.

Fix a partition $\{C_i \mid i<\ka\}$ of $\ka$ with $\size{C_i}=\delta$.
Fix also a bijection $\pi_i : \delta \to C_i$.
Take a sequence $\seq{\beta_i\mid i<\ka}$ in $\mu$
such that $\beta_i \notin M_i \cap \mu$ for $i<\ka$.

Now take a $(V, \mathrm{Add}(\delta, \mu))$-generic $G$,
and for $\gamma<\mu$, let $r_\gamma$ be the $\gamma$-th generic subset of $\delta$ induced by $G$.
Since $\beta_i \notin M_i$, we know that 
$r_{\beta_i}$ is $V[G \cap \mathrm{Add}(\delta, M_i \cap \mu)]$-generic,
hence we have that  $s_i=\pi_i``r_{\beta_i}$ has cardinality $\delta$.
%Moreover, 
%for every $\alpha \in M_i \cap \mu$,
%if $(\dot A_\alpha )_G \cap C_i$ is infinite
%then $(\dot A_\alpha )_G \cap C_i \nsubseteq s_i$ and
%$(\dot A_\alpha )_G) \cap C_i \nsubseteq C_i \setminus s_i$.
Let $B=\bigcup_{i<\ka} s_i$.
We know $B \in [\ka]^\ka$.
We prove that $(\dot A_\alpha)_G \nsubseteq B$ for $\alpha<\mu$.
Fix $\alpha<\mu$, and take a condition
$p \in G$.
Pick a large $i<\ka$ such that
$\alpha,p \in M_i$.
Let $H=G \cap \mathrm{Add}(\delta, M_i \cap \mu)$.
We know $(\dot A_\alpha)_G=(\dot A_\alpha)_H$,
hence $(\dot A_\alpha)_G \in V[H]$.
Since $\bigcup_{j \le i} C_j$ has cardinality $<\ka$ but $(\dot A_\alpha)_G$ has $\ka$,
we can take $k>i$ and $\xi \in (\dot A_\alpha )_G \cap C_k$.
By the definition of $s_k$,
there is a condition $q \in \mathrm{Add}(\delta, \mu \setminus M_i)$
such that, in $V[H]$, 
$q$ forces that $\check \xi \notin \dot s_k$, where $\dot s_k$ is a canonical name for $s_k$.
Thus, in $V$, there is some $p' \le p$ in $\mathrm{Add}(\delta, \mu)$
such that $p' \Vdash \dot A_\alpha \nsubseteq \dot B$, where $\dot B$ is a 
canonical name for $B$.
A similar argument shows that
$(\dot A_\alpha)_G \nsubseteq \ka \setminus B$ for $\alpha<\mu$.
These means that $B, \ka \setminus B \notin (\dot \calU)_G$,
this is a contradiction.
\end{proof}

\begin{note}
We do not need this lemma if  $\cf(\mathfrak{u}(\ka)) \neq \cf(\ka)$ is provable from $\ZFC$,
and it is an open problem asked by  Garti and Shelah \cite{GS}.
\end{note}
%\item Is it possible that $\cf(\frak u(\ka))=\cf(\ka)$?

We are ready to show that monotonicity of the ultrafilter number function can fail.

\begin{prop}\label{3.9}
Let $\ka$ be a cardinal,
and $\mu>\ka$ a strong limit singular cardinal with $\om_1 \le \cf(\mu) <\ka$.
Let $\la<\ka$ be a cardinal with $\cf(\la)=\cf(\mu)$.
If $\ka$ has a $\cf(\la)$-indecomposable uniform ultrafilter,
then $\mathrm{Add}(\om, \mu)$ forces that
$\mathfrak{u}(\ka)<\mathfrak{u}(\la)$.
\end{prop}
\begin{proof}
In the generic extension,
we know $\mathfrak{u}(\ka)\le \mu \le \mathfrak{u}(\la)$
by Theorem \ref{3} and Lemma \ref{10}.
By Lemma \ref{11}, we have $\mu \neq \mathfrak{u}(\la)$,
so $\mathfrak{u}(\ka)<\mathfrak{u}(\la)$.
\end{proof}
We note that in fact $\mathfrak{u}(\ka)=\mu$ holds in the resulting model of this proposition,
and this was already remarked by Benhamou \cite[Proposition 5.12]{Ben}.

The following theorems are typical applications of this proposition:
\begin{thm}\label{4.4}
If $\ka$ is measurable and $\mu>\ka$ is a strong limit singular cardinal with
cofinality $\om_1$, then $\Add(\om, \mu)$ forces $\mathfrak{u}(\ka)<\mathfrak{u}(\om_1)$.
\end{thm}
We note that $\omega_1$ can be replaced by any cardinal $<\ka$ with uncountable cofinality.

\begin{thm}\label{4.5}
Suppose $\ka$ is measurable and $\calU$ a normal measure over $\ka$.
Let $\bbP_\calU$ be the Prikry forcing associated with $\calU$.
Let $\mu>\ka$ be a strong limit singular cardinal with
cofinality $\om_1$.
Then $\bbP_\calU \times \Add(\om, \mu)$ forces that $\ka$ is a singular cardinal with countable cofinality and 
$\mathfrak{u}(\ka)<\mathfrak{u}(\om_1)$.
\end{thm}

By using Theorem \ref{BM theorem} we have:
\begin{thm}\label{3.10}
Relative to a certain large cardinal assumption,
it is consistent that
$\mathfrak{u}(\om_{\om+1})<\mathfrak{u}(\om_1)$.
\end{thm}

If $\ka$ is strongly compact,
then every regular cardinal $\la \ge \ka$ has a 
$\ka$-complete uniform ultrafilter. By this theorem, we also have:
\begin{thm}\label{3.11}
Suppose $\ka$ is strongly compact.
Let $\mu>\ka^{+\om_1}$ be a strong limit singular cardinal with $\cf(\mu)=\om_1$.
Then $\mathrm{Add}(\om, \mu)$ 
forces 
that $\mathfrak{u}(\ka^+)<\mathfrak{u}(\om_1)$ and 
$\mathfrak{u}(\ka^{+\om_1+1})<\mathfrak{u}(\ka^{+\om_1})$.
\end{thm}
$\om_1$ in this theorem can be replaced by any regular uncountable cardinal $\mathrm{<}\ka$.
We also note that $\ka^{+\om_1}$ is not strong limit in the resulting model of this theorem.
In Section \ref{sec7}, we prove that $\om_1$ cannot be replaced by $\om$,
and if $\mathfrak{u}(\ka^{+\om_1+1})<\mathfrak{u}(\ka^{+\om_1})$, then
$\ka^{+\om_1}$ is never strong limit.

\section{Easy $\ZFC$-results}

In this section we shall establish some easy but useful observations about
the ultrafilter number function.

The following lemma is folklore but will be used frequently.
\begin{lemma}\label{4.1}
Let $\la<\ka$ be cardinals,
and $\calU$ a $\la$-decomposable ultrafilter over $\ka$.
Then $\mathfrak{u}(\la) \le \chi(\calU)$.
\end{lemma}
\begin{proof}
Since $\calU$ is $\la$-decomposable, by Lemma \ref{2.8},
there is $f:\ka \to \la$ such that $f_*(\calU)$ is a uniform ultrafilter over $\la$
with $\chi(f_*(\calU)) \le \chi(\calU)$.
Thus we have $\mathfrak{u}(\la) \le \chi(f_*(\calU)) \le \chi(\calU)$.
\end{proof}

The next lemma is just the contraposition of Lemma \ref{4.1}.
\begin{cor}\label{4.2}
Let $\la<\ka$ be cardinals with $\mathfrak{u}(\ka)<\mathfrak{u}(\la)$.
Then every uniform ultrafilter $\calU$ over $\ka$ with $\chi(\calU)=\mathfrak{u}(\ka)$ 
is $\la$-indecomposable.
%
%with $\chi(\calU)=\mathfrak{u}(\ka)$ is $\la$-indecomposable.
%In particular $\ka$ has a $\la$-indecomposable uniform ultrafilter.
\end{cor}

By these observations, we can obtain some restrictions of the failure
of monotonicity.
%First, $\mathfrak{u}(\om)$ is the minimum value of the ultrafilter numbers.
\begin{prop}\label{5.3+}
$\mathfrak{u}(\om) \le \mathfrak{u}(\ka)$ for every cardinal $\ka$.
\end{prop}
\begin{proof}
If not, then $\ka$ has an $\om$-indecomposable uniform ultrafilter $\calU$
by Corollary \ref{4.2}.
Then $\calU$ is in fact $\sigma$-complete,
so there is a measurable cardinal $\la \le \ka$.
Hence we have $\mathfrak{u}(\om) \le 2^\om<\la \le \ka \le \mathfrak{u}(\ka)$,
this is a contradiction.
\end{proof}

\begin{prop}\label{21}
\begin{enumerate}
\item If $\ka$ is regular,
then $\mathfrak{u}(\ka) \le \mathfrak{u}(\ka^+)$.
\item If $\la$ is regular, $\ka>\la$ a cardinal,
and $\mathfrak{u}(\ka)<\mathfrak{u}(\la)$, then
$\la^{+\om} \le \ka$.
%\item If $\la$ is singular, $\ka >\la^+$ is a cardinal,
%and $\mathfrak{u}(\ka)<\mathfrak{u}(\la) \le \mathfrak{u}(\la^+)$, then
%$\la^{+\om} \le \ka$.
\end{enumerate}
\end{prop}
\begin{proof}
(1). If $\mathfrak{u}(\ka^+)<\mathfrak{u}(\ka)$,
then $\ka^+$ has a $\ka$-indecomposable uniform ultrafilter,
but this is impossible by Lemma \ref{2.4}.

(2). By (1), we have $\mathfrak{u} (\la) \le \mathfrak{u}(\la^+) \le \mathfrak{u} (\la^{++}) \le \cdots$,
hence $\la^{+\om} \le \ka$.

\end{proof}
We will prove that $\mathfrak{u}(\om_\om) < \mathfrak{u}(\om_1)$ is consistent (Corollary \ref{4.14}).
In this sense, $\la^{+\om} \le \ka$ in (2) of this corollary cannot be
strengthened to $\la^{+\om+1} \le \ka$.

If $\ka$ is singular,  every uniform ultrafilter over $\ka$ is 
$\cf(\ka)$-decomposable.
\begin{cor}
If $\ka$ is singular, then $\mathfrak{u}(\cf(\ka)) \le \mathfrak{u}(\ka)$.
\end{cor}

\section{Models with the failure of monotonicity 2}
In Section \ref{sec1}, we constructed some models with $\la<\ka$ but
$\mathfrak{u}(\ka)<\mathfrak{u}(\la)$.
Our main tools were Raghavan and Shelah's theorem \ref{3} and
Lemma \ref{10}, \ref{11}.
However, we cannot use these tools when $\cf(\la)=\om$.
In this section, we show that it is possible that
$\mathfrak{u}(\ka)<\mathfrak{u}(\om_\om)$ for some $\ka>\om_\om$.

The following lemma was proved by Donder \cite{Donder},
but we give a proof for the reader's convenience.
\begin{lemma}[Donder \cite{Donder}]\label{5.1}
Let $\ka$ be a singular cardinal with countable cofinality,
and $\seq{\ka_n \mid n<\om}$ a strictly increasing sequence of regular cardinals
with limit $\ka$.
If $\calU$ is a uniform ultrafilter over $\ka$ such that
$\calU$ is not $(\om, \lambda)$-regular for some $\lambda<\ka_0$,
then there is $f:\ka \to \ka$ such that
$f_*(\calU)$ is a uniform ultrafilter over $\ka$ and
for every $C \subseteq \ka$, if $C \cap \ka_n$ is a club in $\ka_n$ for every $n<\om$,
then $C \in f_*(\calU)$.
\end{lemma}
\begin{proof}
We recall some definition and notations.
For $f,g \in {}^\ka \ka$,
define $f \le g$ if $f(\alpha) \le g(\alpha)$ for every $\alpha<\ka$,
and $f<_\calU g$ if $\{\alpha<\ka \mid f(\alpha)<g(\alpha)\} \in \calU$.

Let $\calF$ be the family of all $f:\ka \to \ka$ such that:
\begin{enumerate}
\item $f$ is monotone, that is, if $\alpha\le \beta$ then $f(\alpha) \le f(\beta)$. 
\item $\sup(f``\ka_n)=\ka_n$ for every $n<\om$.
\end{enumerate}
We claim that there is $f \in \calF$ which is $<_\calU$-minimal in $\calF$,
that is, if $g \in \calF$ then $\{\alpha<\ka \mid g(\alpha) \ge f(\alpha)\} \in \calU$.
Such a $<_\calU$-minimal function is as wanted;
This can be verified as follows.
Take  $C \subseteq \ka$ such that $C \cap \ka_n$ is a club in $\ka_n$ for every $n<\om$.
Define $g:\ka \to \ka$ by $g(\alpha)=\sup(C \cap \alpha)$.
Clearly $g \in \calF$ and $g \le \mathrm{id}$.
Since $g \circ f \in \calF$ and $f\le_U g \circ f$,
we have $X=\{\alpha<\ka \mid g(\alpha)=\alpha\} \in f_*(\calU)$.
By the definition of $g$, we also have 
$X \subseteq (C \setminus (\min(C)+1)) \cup \{\ka_n \mid n<\om\}$.
On the other hand, since $f_*(\calU)$ is uniform, we have $\{\ka_n \mid n<\om\} \notin \calU$,
hence $C \in f_*(\calU)$.

We show that $\calF$ has a $<_\calU$-minimal element.
Suppose to the contrary that such an element does not exist.
By induction on $\xi<\lambda$,
we take $f_\xi \in \calF$ so that:
\begin{enumerate}
\item[(a)] $f_{\xi+1} <_\calU f_{\xi}$.
\item[(b)] $f_{\xi} \le f_{\eta}$ whenever $\eta<\xi<\delta$.
\end{enumerate}
For a successor step,
if $f_\xi$ is defined,
there is $g \in \calF$ with $g<_\calU f_\xi$ by our assumption.
Define $f_{\xi+1} \in {}^\ka \ka$ by $f_{\xi+1}(\alpha)=g(\xi)$ if $g(\alpha)<f_\xi(\alpha)$,
otherwise let $f_{\xi+1}(\alpha)=f_{\xi(\alpha)}$.
One can check that $f_{\xi+1} \in \calF$, $f_{\xi+1} \le f_\xi$, and
$f_{\xi+1}<_\calU f_\xi$.
For a limit $\xi<\lambda$,
define $f_\xi(\alpha)=\min \{f_\zeta(\alpha) \mid \zeta<\xi\}$.
It is not hard to check that $f_\xi \in \calF$ and
$f_\xi \le f_\zeta$ for $\zeta<\xi$.

Finally, for $\xi<\delta$,
let $X_\xi=\{\alpha<\ka \mid f_{\xi+1}(\alpha) <f_\xi(\alpha) \} \in \calU$.
It is easy to check that $\{X_\xi \mid \xi<\delta\}$ witnesses that
$U$ is $(\om,\lambda)$-regular, 
this is a contradiction.
\end{proof}

The following can be seen as a diagonal version of Prikry and Silver's theorem (Theorem \ref{PS theorem}).

\begin{prop}\label{7.2}
Let $\ka$ be a singular cardinal with countable cofinality, and $\lambda<\ka$ a regular
cardinal with $\la^\om=\la$.
Suppose there are sequences $\seq{\ka_n \mid n<\om}$ and
$\seq{S_n \mid n<\om}$ satisfying the following:
\begin{enumerate}
\item $\seq{\ka_n \mid n<\om}$ is an increasing sequence of regular cardinals with limit $\ka$.
\item For each $n<\om$, $S_n$ is a non-reflecting stationary subset of
$\{\alpha<\ka_n \mid \cf(\alpha)=\lambda\}$.
\item For every enough large regular cardinal $\theta$,
the set $\{M \in [H_\theta]^\lambda\mid  {}^\om M \subseteq M$,
and $\forall n<\om (\sup(M \cap \ka_n) \in S_n)\}$ is stationary in $[H_\theta]^\la$.
\end{enumerate}
Then there is no $\lambda$-indecomposable uniform ultrafilter over $\ka$.
\end{prop}
Note that the assumption (3) of this proposition is different from
the notion of \emph{mutually stationary},
which only requires the set $\{M \in [H_\theta]^\lambda\mid$
$\forall n<\om (\sup(M \cap \ka_n) \in S_n)\}$ is stationary in $[H_\theta]^\la$;
$M$ need not to be closed under $\om$-sequences.

\begin{proof}
Fix such $\seq{\ka_n \mid n<\om}$ and $\seq{S_n \mid n<\om}$.
We may assume that $\lambda<\ka_0$ and $S_n \cap S_m=\emptyset$ whenever $n<m<\om$.
For each $\beta<\ka$, let $n(\beta)$ be the minimal $n<\om$ with $\beta<\ka_n$.

For $n<\om$ and $\alpha \in S_n$,
fix a club $c_\alpha \subseteq \alpha$ with order type $\lambda$.
Let $\seq{\xi_i^\alpha \mid i<\lambda}$ be the increasing enumeration of
$c_\alpha$.
For the next well-known theorem, we use the fact that the sets $S_n$'s are non-reflecting.
See Lemma 2.12 in Eisworth \cite{Eisworth} for the proof.
\begin{thm}\label{5.3}
For every $n<\om$ and $\beta<\ka_n$,
there is a function $F_\beta : S_n \cap \beta \to \lambda$ such that
whenever $\alpha_0, \alpha_1 \in S_n \cap \beta$ with $\alpha_0 \neq \alpha_1$,
the set $\{\xi_i^{\alpha_0} \mid F_\beta(\alpha_0)<i<\lambda\}$
is disjoint from $\{\xi_i^{\alpha_1} \mid F_\beta(\alpha_1)<i<\lambda\}$.
\end{thm}

Now suppose the contrary that there is a $\lambda$-indecomposable uniform ultrafilter $\calU$ over $\ka$,
in particular it is not $(\om, \la)$-regular. By Lemma \ref{2.8} and \ref{5.1},
we may assume that for every $ t \in \prod_{n<\om} \ka_n$, the set $\bigcup_{n<\om} [t(n),\ka_n)$ is in $\calU$.
For $t \in \prod_{n<\om} S_n$,
define $\lambda(t)<\lambda$ as follows:
By the choice of $\calU$,
we have that $X_t=\{\beta<\ka \mid \beta>t(n(\beta))\} \in \calU$.
Define $G_t:X_t \to \lambda$
by $G_t(\beta)=F_\beta(t(n(\beta))$.
Since $\calU$ is $\lambda$-indecomposable,
there is $\lambda(t)<\lambda$ such that
$Y_t=\{\beta \in X_t \mid G_t(\beta)=F_\beta(t(n(\beta))) <\lambda(t)\} \in \calU$.

\begin{claim}\label{5.4}
For $t, t' \in \prod_{n<\om} S_n$,
if $\{n<\om \mid t(n) \neq t'(n)\}$ is cofinite,
then there are infinitely many $n<\om$ such that
$\{\xi_i^{t(n)} \mid \lambda(t)<i<\lambda\} \cap
\{\xi_i^{t'(n)} \mid \lambda(t')<i<\lambda\}=\emptyset$.
\end{claim}
\begin{proof}
We know that $Y_t \cap Y_{t'} \in \calU$.
For a given $m<\om$, 
since $\calU$ is uniform and 
the set $\{n<\om \mid t(n) \neq t(n')\}$ is cofinite,
we have that the set $\{\beta<\ka \mid n(\beta)>m, t(n(\beta))\neq t'(n(\beta))\}$ is in $\calU$.
Thus we can find $\beta \in Y_t \cap Y_{t'}$ 
such that $m<n(\beta)$, $t(n(\beta)) \neq t'(n(\beta))$, and $t(n(\beta)),t'(n(\beta))<\beta$.
By the choice of $F_\beta$, we have
$\{\xi_i^{t(n(\beta))} \mid F_\beta(t(n(\beta)))<i<\lambda\} \cap
\{\xi_i^{t'(n(\beta))} \mid F_\beta(t'(n(\beta)))<i<\lambda\}=\emptyset$.
Since $\beta \in Y_t \cap Y_{t'}$, 
we know $F_\beta(t(n(\beta)))<\lambda(t)$ and 
$F_\beta(t'(n(\beta))) <\lambda(t')$.
Hence 
$\{\xi_i^{t(n(\beta))} \mid \lambda(t)<i<\lambda\} \cap
\{\xi_i^{t'(n(\beta))} \mid \lambda(t')<i<\lambda\}=\emptyset$.
\end{proof}

Fix a large regular $\theta$.
By the choice of the  $S_n$'s,
we can find $M \in [H_\theta]^\lambda$ containing all relevant objects such that $M \prec H_\theta$,
${}^\om M \subseteq M$,
 and $\forall n<\om (\sup(M \cap \ka_n) \in S_n)$.
Define $t \in \prod_n S_n$ by $t(n)=\sup(M \cap \ka_n)$.
Note that $\sup(c_{t(n)} \cap M \cap \ka_n)=\sup(M \cap \ka_n)$ because
$c_{t(n)}$ is a club in $t(n)$ and ${}^\om M \subseteq M$.
For $n<\om$,
pick $\xi_n \in M \cap \{\xi_i^{t(n)} \mid \lambda(t)<i<\lambda\}$.
We know $\seq{\xi_n \mid n<\om} \in M$ by the closure property of $M$.
By the elementarity of $M$,
there is $t' \in M \cap \prod_n S_n$ such that
$\xi_n \in \{\xi_i^{t'(n)} \mid \lambda(t')<i<\lambda\}$ for every $n<\om$.
Clearly $t(n) \neq t'(n)$ for every $n<\om$,
hence there is $n<\om$ with 
$\{\xi_i^{t(n)} \mid \lambda(t)<i<\lambda\} \cap
\{\xi_i^{t'(n)} \mid \lambda(t')<i<\lambda\}=\emptyset$
by Claim \ref{5.4}.
However this is impossible since
$\xi_n \in 
\{\xi_i^{t(n))} \mid \lambda(t)<i<\lambda\} \cap
\{\xi_i^{t'(n))} \mid \lambda(t')<i<\lambda\}$.
\end{proof}

To show that the assumption of Proposition \ref{7.2} is consistent,
we use a forcing notion which adds a non-reflecting stationary subset.

For a poset $\bbP$ and an ordinal $\alpha$,
let $\Gamma_\alpha(\bbP)$ denote the following two players game of length $\alpha$:
At each inning, Players I and II choose conditions of $\bbP$ alternately with
$p_0 \ge q_0 \ge p_1 \ge q_1 \ge \cdots$, but if $\beta<\alpha$ is limit,
at the $\beta$-th inning, Player I does not move and
only Player II choose a condition $q_\beta$ which is a lower bound of the partial play $\langle p_\xi, q_\zeta \mid \xi,\zeta<\beta$,
$\xi=0$ or successor$\rangle$ (if it is possible):
\begin{center}
\begin{tabular}{c||c|c|c|c|c|c}
 & $0$ & $1$ & $\cdots$       & $\omega$ & $\omega+1$ & $\cdots$ \\
\hline
I & $p_0$ & $p_1$ & $\cdots$ &         & $p_{\omega+1}$ & $\cdots$  \\
\hline
II & $q_0$ & $q_1$ & $\cdots$ & $q_{\omega}$ & $q_{\omega+1}$ &  $\cdots$\\

\end{tabular}
\end{center}

Player II \emph{wins} if I and II could choose their moves at each inning,
otherwise, I \emph{wins}.

$\bbP$ is \emph{$\alpha$-strategically closed} if Player II has a winning strategy
in the game $\Gamma_\alpha(\bbP)$.
If $\ka$ is a cardinal and $\bbP$ is $\ka$-strategically closed,
then  forcing with $\bbP$
does not add new $<\ka$-sequences.

\begin{prop}\label{6.5+}
Let $\ka$ be a cardinal,
and $\seq{\bbP_n, \dot \bbQ_n \mid  n<\om}$ be an $\om$-stage iteration such that
 $\Vdash_{\bbP_n}$``\,$\dot \bbQ_n$ is $\ka$-strategically closed'' for every $n<\om$.
Then the inverse limit of the $\bbP_n$'s is $\ka$-strategically closed.
\end{prop}

For a regular uncountable cardinal $\mu$ and regular $\lambda<\mu$,
let $\mathbb{S}_{\mu,\lambda}$ be the
poset of all bounded subsets $p \subseteq \{\alpha<\mu \mid \cf(\alpha)=\lambda\}$
such that $p$ has a maximal element $\max(p)$ and for every $\beta \le \max(p)$ with $\cf(\beta)>\om$,
$p \cap \beta$ is non-stationary in $\beta$.
Define $p \le q$ if $p$ is an end-extension of $q$.

Let $\mathbb{S}=\mathbb{S}_{\mu,\lambda}$.
The following is easy to check:
\begin{lemma}
$\size{\mathbb{S}} =2^{<\mu}$, so $\mathbb{S}$ satisfies the $(2^{<\mu})^+$-c.c.
\end{lemma}

% Then the following hold:

The following lemma is well-known, and we only sketch the  proof.
\begin{lemma}
\begin{enumerate}
\item $\mathbb{S}$ is $\lambda^+$-closed.
\item $\mathbb{S}$ is $\mu$-strategically closed, hence forcing with $\mathbb S$ does not add new $\mathop{<}\mu$-sequences. \item If $G$ is $(V, \mathbb{S})$-generic,
then $\bigcup G$ is a non-reflecting stationary subset of $\{\alpha<\mu\mid \cf(\alpha)=\lambda\}$.
\end{enumerate}
\end{lemma}
\begin{proof}[Sketch of the proof]
(1). For the $\lambda^+$-closedness,
take a descending sequence $\seq{p_\xi \mid \xi<\alpha}$ for some $\alpha<\lambda^+$.
Let $\gamma =\sup_{\xi<\alpha} (\max(p_\xi))$.
We know $\cf(\gamma) \le \la$,
hence $\bigcup_{\xi<\alpha} p_\xi$ is non-stationary in $\gamma$.
Let $p=\bigcup_{\xi<\alpha} p_\xi \cup \{\gamma+\lambda\}$.
we can check that $p \in \mathbb{S}$ and $p$ is a lower bound of the $p_\xi$'s.

(2). For a limit $\beta<\mu$ and a partial play $\langle p_\xi, q_\zeta \mid \xi,\zeta<\beta$,
$\xi=0$ or successor$\rangle$,
suppose $q=\bigcup_{\zeta<\beta} q_\zeta$ is non-reflecting. Let $\gamma=\sup(q)$.
Then Player II takes $q \cup \{\gamma+\la\}$ as his move.
This is a winning strategy of Player II.

(3). 
Non-reflectingness is immediate from the definition of $\mathbb{S}$.
To show that $\bigcup G$ is stationary, take a name $\dot C$ for a club in $\mu$.
Take a descending sequence $\seq{p_i \mid i<\la}$ such that
for every $i<\la$,
there is $\alpha_i$ such that $\alpha_i<\max(p_i)<\alpha_{i+1}$
and $p_{i+1} \Vdash \check \alpha_i \in \dot C$.
Let $\alpha=\sup_i \alpha_i$, and $p=\bigcup_{i<\la} p_i \cup \{\alpha\}$.
It is easy to check that $p \in \mathbb S$ and
$p \Vdash \alpha \in \dot C \cap \bigcup \dot G$.
\end{proof}

\begin{prop}\label{6.8}
Let $\ka$ be a singular cardinal with countable cofinality,
and suppose GCH.
Let $\seq{\ka_n \mid n<\om}$ be a strictly increasing sequence of
regular uncountable cardinals with limit $\ka$, and $\lambda$ a regular uncountable cardinal with
$\la<\ka_0$.
Then there is a poset $\bbP$ 
such that $\size{\bbP}=\ka^+$, 
$\bbP$ preserves all cardinals,
and adds a sequence $\seq{S_n \subseteq \ka_n\mid n<\om}$
satisfying the conditions of Proposition \ref{7.2}.
\end{prop}
\begin{proof}
Fix such $\seq{\ka_n \mid n<\om}$ and $\la$.
Define $\bbP_n$ and $\dot \bbQ_n$ for $n<\om$ as follows.
Let $\bbP_0$ be the trivial poset.
Suppose $\bbP_n$ is defined.
Then $\dot \bbQ_n$ is a $\bbP_n$-name
such that
\[
\Vdash_{\bbP_n} \dot \bbQ_n=\dot{\mathbb{S}}_{\ka_n,\la}.
\]
Let $\bbP_{n+1}=\bbP_n*\dot \bbQ_n$.
It is routine to show that $\size{\bbP_{n+1}} \le 2^{<\ka_{n}}=\ka_n$,
so it satisfies the $(\ka_n)^+$-c.c.
Let $\bbP$ be the inverse limit of the $\bbP_n$'s.
We check that $\bbP$ witnessing the theorem.

A standard argument shows $\size{\bbP}=\ka^\om=\ka^+$ and $\bbP$ is $\la^+$-closed.
By Proposition \ref{6.5+}, for each $n<\om$, $\bbP$ can be factored as
a forcing product $\bbP_n*\dot\bbQ_n* \dot \bbP_{tail}$
such that $\Vdash_{\bbP_n*\dot \bbQ_n}$``$\dot \bbP_{tail}$ is $\ka_{n+1}$-strategically closed''.
This factorization shows that forcing with $\bbP$ preserves all cofinalities $\mathop{<}\ka$.
Since $\size{\bbP}=\ka^+$, it preserves all cofinalities $>\ka^+$.
We show that $\bbP$ preserves $\ka^+$. If $(\ka^+)^V$ is collapsed in $V^\bbP$,
since $\ka$ is singular, the cofinality of $(\ka^+)^V$ is $ \mathop{<}\ka$ in $V^{\bbP}$.
Hence there is a cofinal map $f$ from some $\alpha<\ka$ into
$(\ka^+)^V$. By the above factorization, this $f$ was added by $\bbP_n$ for some $n<\om$,
so $(\ka^+)^V$ is collapsed in $V^{\bbP_n}$.
This is impossible because the cardinality of $\bbP_n$ is $<\ka$.

Take  a $(V,\bbP)$-generic $G$, and
let $S_n$ be a generic subset of $\ka_n$ induced by $G$.
By the above factorization, we have that $S_n$ is a non-reflecting stationary subset of $\{\alpha<\ka_n \mid \cf(\alpha)=\la\}$
in $V[G]$.

Finally, we see that the sequence $\seq{S_n \mid n<\om}$ satisfies the 
the conditions of Proposition \ref{7.2}. To do this, 
return to $V$. Let $\theta$ be a sufficiently large regular cardinal
and $N=H_\theta$. Let $\dot S_n$ be a canonical name for $S_n$.
Take a $\bbP$-name $\dot f$ such that $p \Vdash \dot f:[\check N]^{<\om} \to \check N$.
It is enough to show that there is $X \subseteq N$ and $q \le p$ such that
$\size{X}=\la \subseteq X$, ${}^\om X\subseteq X$,
and 
\[
q \Vdash \dot f``[\check X]^{<\om} \subseteq \check X, \sup(\check X \cap \check \ka_n) \in
\dot S_n \text{ for $n<\om$}.
\]
Take a large regular cardinal $\chi>\theta$ and 
$M \prec H_\chi$ containing $N,p, \seq{\dot S_n\mid n<\om}, \dot f$ and other relevant objects such that
$\size{M} =\la \subseteq M$ and ${}^{<\la}M \subseteq M$.
Let $X=M \cap H_\theta=M \cap N$.
Note that $\size{X}=\la$, ${}^\om X \subseteq X$, and 
$\cf(\sup(X \cap \ka_n))=\la$ for every $n<\om$.

Since $\bbP$ is $\la^+$-closed, $\size{M}=\la \subseteq M$, and ${}^{<\la}M \subseteq M$,
it is easy to take a descending sequence $\seq{q_i \mid i<\la}$ in $\bbP$ such that:
\begin{enumerate}
\item $q_0 \le p$ and $q_i \in M \cap \bbP$ for every $i<\la$.
\item For every dense set $D \in M$ in $\bbP$,
there is $i<\la$ with $q_i \in D \cap M$.
\end{enumerate}
By the condition (2) and the elementarity of $M$, we have:
\begin{enumerate}
\item[(3)] For every $x \in [M \cap N]^{<\om}$,
there is $i<\la$ and $y \in M \cap N (=X)$
such that $q_{i} \Vdash \dot f(\check x)=\check y$.
\item[(4)] For $i<\la$ and $n<\om$,
there is $j<\la$ and $\alpha_{i,n} \in M \cap N \cap \ka_n (=X \cap \ka_n)$
such that $q_{j}\restriction n \Vdash_{\bbP_n} \max(q_i(n))=\check \alpha_{i,n}$.
\end{enumerate}

\begin{claim}
$\sup(X \cap \ka_n)=\sup_{i<\la} \alpha_{i,n}$ for every $n<\om$.
\end{claim}
\begin{proof}
Since $\alpha_{i,n} \in M \cap N  \cap \ka_n=X \cap \ka_n$,
we have $\sup(X \cap \ka_n) \ge \sup_{i<\la} \alpha_{i,n}$.
For the converse, take $\alpha \in X \cap \ka_n$.
Then the set $D=\{r \in \bbP  \mid r\restriction n \Vdash_{\bbP_n} \max(r(n)) \ge \check \alpha\}$ is dense in $\bbP$ and is in $M$. By the condition (4)
there is $i<j<\la$ such that $q_i \in M \cap D$ and
$q_j\restriction n \Vdash_{\bbP_n} \max(q_i(n))=\check \alpha_{i,n}$.
Thus $\alpha_{i,n} \ge \alpha$, so $\sup(X \cap \ka_n) \le \sup_{i<\la} \alpha_{i,n}$.
\end{proof}

Take a function $q$ on $\om$ as follows:
$q(n)$ is a $\bbP_n$-name such that $\Vdash_{\bbP_n}$``$q(n)=\bigcup_{i<\la} q_i(n) \cup 
\{\sup (\check X \cap \check \ka_n)\}$.
Finally we show that
$q$ is a lower bound of the $q_i$'s, and
$q \Vdash \dot f``[\check X]^{<\om} \subseteq \check X, \sup(\check X \cap \check \ka_n) \in
\dot S_n$ for $n<\om$.
\begin{claim}
$q \in \bbP$ and $q$ is a lower bound of $\{q_i \mid i<\la\}$.
\end{claim}
\begin{proof}
It is enough to show that 
$q \restriction n \in \bbP_n$ and
$q \restriction n$ is a lower bound of $\{ q_i \restriction n \mid i<\la\}$
for every $n<\om$.  We prove this by induction on $n<\om$.
The case $n=0$ is trivial.
Suppose $q \restriction n \in \bbP_n$
and is a lower bound of $\{q_i \restriction n \mid i<\la\}$.
Then $q \restriction n \Vdash_{\bbP_n}$``$\{q_i(n)\mid i<\la\}$ is 
a descending sequence in $\dot \bbQ_n=\dot{ \mathbb{S}}_{\ka_n, \la}$''.
By the condition (4) and the previous claim,
we know that  $q \restriction n \Vdash_{\bbP_n}$``
$\sup(\{\max(q_i(n))\mid i<\la\})=\sup(\check X \cap \check \ka_n)$''.
Since $\cf(X \cap \ka_n)=\la$,
we also have $q \restriction n \Vdash_{\bbP_n}$``$\bigcup_{i<\la} q_i(n)$ is non-stationary in $\sup(\check X \cap \check \ka_n)$''.
These observations show that
$q \restriction \Vdash_{\bbP_n}$``$q(n)=\bigcup_{i<\la} q_i(n) \cup 
\{\sup (\check X \cap \check \ka_n)\} \in \dot \bbQ_n$'',
so $q \restriction( n+1) \in \bbP_{n+1}$.
By the definition of $q$,
we also have $q \restriction(n+1)$ is a lower bound of 
 $\{q_i \restriction(n+1) \mid i<\la\}$.
\end{proof}
Now, since $q$ is a lower bound of $\{q_i \mid i<\la\}$,
we know that $q \Vdash \dot f``[\check X]^{<\om} \subseteq \check X$ by the condition (3).
Moreover
$q \Vdash \sup(\check X \cap \check \ka_n) \in \dot S_n$ for every $n<\om$.
This completes the proof.

\end{proof}

By this proposition, 
we can construct a model in which there is a measurable cardinal and
there is no $\om_1$-indecomposable uniform ultrafilter over $\om_\om$;
Suppose GCH holds and there is a measurable cardinal $\nu$.
Let $\bbP$ be a poset from Proposition \ref{6.8} for $\ka=\om_\om$ and $\la=\om_1$.
In $V^\bbP$ there is no $\om_1$-indecomposable uniform ultrafilter over $\om_\om$
by Proposition \ref{7.2}. The cardinality of $\bbP$ is $\mathop{<}\nu$,
hence $\nu$ remains a measurable cardinal in $V^\bbP$ by
Levy-Solovay Theorem (e.g., see Proposition 10.15 in Kanamori \cite{Kanamori2}).

\begin{thm}\label{7.5}
Suppose there is no $\om_1$-indecomposable uniform ultrafilter over $\om_\om$.
Let $\ka$ be a measurable cardinal,
and $\mu>\ka$ a strong limit singular cardinal with cofinality $\om_1$.
\begin{enumerate}
\item $\Add(\om, \mu)$ forces $\mathfrak{u}(\ka)<\mathfrak{u}(\om_\om)$.
\item If $\calU$ is a normal measure over $\ka$ and $\bbP_\calU$ is the Prikry forcing associated with $\calU$,
then $\bbP_\calU \times \Add(\om, \mu)$ forces that
$\ka$ is a singular cardinal with countable cofinality and $\mathfrak{u}(\ka)<\mathfrak{u}(\om_\om)$.
\end{enumerate}
\end{thm}
\begin{proof}
(1). Take a $(V, \Add(\om, \mu))$-generic $G$.
In $V[G]$, we have $\mathfrak{u}(\ka)<\mathfrak{u}(\om_1)$ by Theorem \ref{4.4}.
We claim that, in $V[G]$, $\om_\om$ does not have an $\om_1$-indecomposable uniform ultrafilter.
Suppose to the contrary that there is an $\om_1$-indecomposable uniform ultrafilter $\calU$.
Let $\dot \calU$ be a name for $\calU$.
In $V$, let $\calF=\{X \subseteq \om_\om \mid\, \Vdash \check X \in \dot \calU\}$.
$\calF$ is a uniform filter over $\om_\om$,
and since $\Add(\om, \mu)$ has the c.c.c. and
$\calU$ is $\om_1$-indecomposable,
for every $f:\om_\om \to \om_1$
there is $\gamma<\om_1$ such that
$\{\alpha \mid f(\alpha) \le \gamma\} \in \calF$.
This shows that $\calF$ is $\om_1$-indecomposable in $V$.
By Lemma \ref{2.3},  every ultrafilter extending $\calF$ is an $\om_1$-indecomposable uniform ultrafilter,
this is a contradiction.

Now, in $V[G]$, every uniform ultrafilter over $\om_\om$ is $\om_1$-decomposable.
By Corollary \ref{4.2}, we have that $\mathfrak u(\om_1) \le \mathfrak u(\om_\om)$,
hence $\mathfrak u(\ka)<\mathfrak u(\om_1) \le \mathfrak u(\om_\om)$.

(2) follows from a similar argument.
\end{proof}
We also have the following theorem.
We leave the proof to the reader.

\begin{thm}\label{5.7}
Suppose there is no $\om_1$-indecomposable uniform ultrafilter over $\om_\om$.
Let $\ka$ be a strongly compact cardinal,
and $\mu>\ka$ a strong limit singular cardinal with cofinality $\om_1$.
Then $\Add(\om, \mu) $ forces that
$\mathfrak u(\ka^+)<\mathfrak u(\om_\om)$ and $\mathfrak{u}(\ka^{+\om+1})<\mathfrak{u}(\om_\om)$.
\end{thm}

We will show  that
both
$\mathfrak u(\ka^{+\om}) <\mathfrak{u}(\om_\om)$
and $\mathfrak{u}(\ka^{+\om_1})<\mathfrak{u}(\om_\om)$ are consistent as well.
See Corollary \ref{4.14} and Proposition \ref{7.19}.

\section{Monotonicity at a singular cardinal and its successor}\label{sec7}
We had seen that $\mathfrak{u}(\ka) \le \mathfrak{u}(\ka^+)$ always holds if $\ka$ 
is regular. 
In contrast with a regular cardinal, if $\ka$ is singular then it is possible that $\mathfrak{u}(\ka^+)<\mathfrak{u}(\ka)$; 
By Theorem \ref{3.11}, it is consistent that $\fraku(\ka^{+\om_1+1})<\frak{u}(\ka^{+\om_1})$.
Nonetheless, we will show there are some restrictions of the failure of
monotonicity at singular and its successor,
and prove that if $\ka$ is singular and $\fraku(\ka^+)<\frak{u}(\ka)$,
then the cofinality of $\ka$ must be uncountable and $\ka$ is never strong limit.

To accomplish that, we will make use of Shelah's PCF theory.
Let us start this section by presenting some basic definitions and facts.
See Shelah \cite{Shelah}, Abraham and Magidor \cite{Abraham}, and Eisworth \cite{Eisworth} for details.

\begin{thm}[Shelah, see, e.g., Theorem 2.23 and 2.26 in \cite{Eisworth}]
Let $\ka$ be a singular cardinal.
Then there is an increasing sequence $\seq{\ka_\xi \mid \xi<\cf(\ka)}$ of regular cardinals with limit $\ka$
such that
$\prod_{\xi} \ka_\xi$ admits a scale $\seq{f_i \mid i<\ka^+}$ of length $\ka^+$,
 that is, the following hold:
\begin{enumerate}
\item $f_i \in \prod_\xi \ka_\xi$.
\item For $i<j<\ka^+$, $f_i\le^* f_j$ holds, that is, there is $\zeta<\cf(\ka)$
with $f_i(\xi) \le f_j(\xi)$ for every $\xi \ge \zeta$.
\item For every $g \in \prod_\xi \ka_\xi$,
there is $i<\ka^+$ with $g \le^* f_i$.
\end{enumerate}
\end{thm}

We also use the following Kanamori's result.
\begin{thm}[Kanamori, Corollary 2.4 in \cite{Kanamori2}]\label{25}
Let $\ka$ be a singular cardinal.
Then every uniform ultrafilter over $\ka^+$
is $(\ka, \ka^+)$-regular.
\end{thm}

Now we prove the following theorem, which plays a  crucial role in this section and
is interesting in its own right.
\begin{thm}\label{27}
Let $\ka$ be a singular cardinal and 
$\calU$ an ultrafilter over a set $S$.
If $\calU$ is $\ka^+$-decomposable and $\cf(\ka)$-decomposable,
then $\calU$ is $\ka$-decomposable as well.
\end{thm}
\begin{proof}
First we check that $\calU$ is $(\ka, \ka^+)$-regular.
Since $\calU$ is $\ka^+$-decomposable,
we can find $f:S \to \ka^+$
such that $f_*(\calU)$
is a uniform ultrafilter over $\ka^+$ by Lemma \ref{2.8}.
By Theorem \ref{25}, $f_*(\calU)$ is $(\ka, \ka^+)$-regular.
Fix $\calF \subseteq f_*(\calU)$ which witnesses that 
$f_*(\calU)$ is $(\ka, \ka^+)$-regular.
Then it is easy to show that $\{f^{-1}(X) \mid X \in \calF\}$
also witnesses that $\calU$ is $(\ka, \ka^+)$-regular.

Since $\calU$ is $(\ka, \ka^+)$-regular,
we can take a family $\{x_s \mid  s \in S\}$ such that
$\size{x_s}<\ka$ and
for every $i<\ka^+$, we have $\{s \in S \mid i \in x_s\} \in \calU$.
Fix an increasing sequence $\seq{\ka_\xi \mid \xi<\cf(\ka)}$ of 
regular cardinals with limit $\ka$ which admits a scale
$\seq{f_i \mid i<\ka^+}$.
Since $\calU$ is $\cf(\ka)$-decomposable,
we can find $g:S \to \{\ka_\xi \mid \xi<\cf(\ka)\}$
such that for every $\eta<\cf(\ka)$,
we have $\{s \in S  \mid g(\alpha)>\ka_\eta\} \in \calU$.

Now we define $h:S \to \ka$ as follows:
Fix $s \in S$.
Take $\xi_s<\cf(\ka)$ such that
$g(s)+\size{x_s}<\ka_{\xi_s}$.
Then the set $\{f_i(\xi_s) \mid i \in x_s \}$ is bounded in $\ka_{\xi_s}$.
So we can take $h(s)<\ka_{\xi_s}$ with $h(s)>f_i(\xi_s)$ for every $i \in x_s$.
We show that $h$ witnesses that $\calU$ is $\ka$-decomposable.
Suppose not. Take $X \in [\ka]^{<\ka}$ such that
$\{s \in S \mid h(s) \in X\} \in \calU$.
Fix $\zeta_0<\cf(\ka)$ with $\size{X}<\ka_{\zeta_0}$.
Then we can find $j<\ka^+$ and $\zeta_1>\zeta_0$
such that $\sup(X \cap \ka_\xi)<f_j(\xi)$ for every $\xi$ with $\zeta_1 \le \xi<\cf(\ka)$.
Pick $s \in S$ with
$g(s)>\ka_{\zeta_1}$, $j \in x_s$, and $h(s) \in X$.
Since $g(s)>\ka_{\zeta_1}$, we have $\ka_{\xi_s} >\ka_{\zeta_1}$.
Hence $f_j({\xi_s})<h(s)<\ka_{\xi_s}$,
however this is impossible since $h(s) \le  \sup(X \cap \ka_{\xi_s})<f_j({\xi_s})$.
\end{proof}

\begin{lemma}\label{5.6}
Let $\ka$ be a cardinal,
and suppose $\ka$ has a $\sigma$-complete uniform ultrafilter $\calU$
with $\chi(\calU)=\mathfrak{u}(\ka)$.
Then there is a uniform ultrafilter $\mathcal V$ over $\ka$
such that $\chi(\mathcal V)=\mathfrak{u}(\ka)$ and $\mathcal{V}$ is not $\sigma$-complete.
\end{lemma}
\begin{proof}
It is sufficient to show that
there is a non-$\sigma$-complete uniform ultrafilter $\mathcal V$ over $\om \times \ka$
with $\chi(\mathcal V) \le \mathfrak{u}(\ka)$.

Fix a uniform ultrafilter $\mathcal W$ over $\om$.
Define the filter $\mathcal V$ over $\om \times \ka$ by:
$X \in \mathcal{V} \iff$ there is $Y \in \mathcal W$ and $Z \in \calU$ with 
$Y \times Z \subseteq X$.
We have that $\chi(\mathcal{V}) \le \chi(\calU) \times 2^\om$.
By the assumption that there is a $\sigma$-complete uniform ultrafilter over $\ka$,
there is a measurable cardinal $ \le \ka$.
Hence we have $2^\om \le \ka \le \chi(\calU)$ and so
$\chi(\calV) \le \chi(\calU) \times 2^\om =\chi(\calU) =\frak{u}(\ka)$.
Thus $\chi(\calV) \le \frak{u}(\ka)$.
We can check $\mathcal V$ is not $\sigma$-complete as follows:
$X_n=[n, \om) \times \ka \in \mathcal V$
but $\bigcap_n X_n=\emptyset$.

Now we have to show that $\mathcal{V}$ is an ultrafilter.
Take $X \subseteq \om \times \ka$,
and we show that $X \in \calV$ or $(\om \times \ka) \setminus X \in \calV$.
Consider the set $Y=\{n<\om \mid \{\alpha<\ka \mid \seq{n, \alpha} \in X \} \in \calU\}$.

Case 1: $Y \in \mathcal W$.
For $n \in Y$, let $Z_n=\{\alpha<\ka \mid \seq{n,\alpha} \in X\} \in \calU$.
Since $\calU$ is $\sigma$-complete,
we have $Z=\bigcap_n Z_n \in \calU$.
Then $Y \times Z \subseteq X$, hence $X \in \mathcal V$.

Case 2: $Y \notin \mathcal W$.
Since $\calU$ and $\mathcal W$ are ultrafilters,
the set 
$Y'=\{n<\om \mid \{\alpha<\ka \mid \seq{n, \alpha} \in \ka \setminus X \} \in \calU\}$ is in $\mathcal W$.
For $n \in Y'$, let $Z_n'=\{\alpha<\ka \mid \seq{n, \alpha} \in \ka \setminus X \} \in \calU$.
Again, since $\calU$ is $\sigma$-complete, we have $Z'=\bigcap_n Z_n' \in \calU$.
Then $Y' \times Z' \subseteq (\om \times \ka) \setminus X$, and $(\om \times \ka) \setminus X \in \mathcal V$.
\end{proof}

Now we can prove that $\mathfrak{u}(\ka) \le \mathfrak{u}(\ka^+)$ always holds if
$\ka$ is a singular cardinal with countable cofinality.

\begin{cor}\label{29}
If $\ka$ is a singular cardinal with countable cofinality,
then $\mathfrak{u}(\ka) \le \mathfrak{u}(\ka^+)$.
\end{cor}
\begin{proof}

Suppose to the contrary that $\mathfrak{u}(\ka) > \mathfrak{u}(\ka^+)$.
By Lemma \ref{5.6},
there is a uniform ultrafilter $\mathcal U$ over $\ka^+$
such that
$\chi(\mathcal U)=\mathfrak{u}(\ka^+)$ and $\mathcal{U}$ is not $\sigma$-complete.
Then $\calU$ is $\ka$-indecomposable by Corollary \ref{4.2}.
However, since $\calU$ is $\ka^+$-decomposable,
$\calU$ is $\om$-indecomposable by Theorem \ref{27},
so $\calU$ is $\sigma$-complete. This is a contradiction.
\end{proof}
We note that, by Theorem \ref{3.11}, the countable cofinality assumption
of this corollary  cannot be eliminated.

\begin{cor}
If $\ka$ is regular, then $\mathfrak{u}(\ka^{+\alpha}) \le \mathfrak{u}(\ka^{+\alpha+1})$ for 
every $\alpha<\om_1$.
\end{cor}
Again, $\mathfrak{u}(\ka^{+\om_1+1})<\mathfrak{u}(\ka^{+\om_1})$ is consistent
by Theorem \ref{3.11}.

Combining Corollary \ref{29} with Theorems \ref{3.10} and \ref{5.7}
we have:
\begin{cor}\label{4.14}
\begin{enumerate}
\item Relative to a certain large cardinal assumption,
it is consistent that $\mathfrak{u}(\om_\om)<\mathfrak{u}(\om_1)$.
\item Relative to a certain large cardinal assumption,
it is consistent that there is a cardinal $\ka$ with
$\mathfrak{u}(\ka^{+\om})<\mathfrak{u}(\om_\om)$.
\end{enumerate}
\end{cor}

Next we prove that $\mathfrak{u}(\ka) \le \mathfrak{u}(\ka^+)$ holds if $\ka$ is a 
strong limit singular cardinal.

\begin{prop}\label{4.15}
	Let $\ka$ be a singular cardinal and $\la<\ka$ a cardinal,
and suppose $\mu^\la \le \ka^+$ for every $\mu<\ka$. 
If $\ka^+$ has a $\ka$-indecomposable uniform ultrafilter,
then $\ka^\la \le \ka^+$.
\end{prop}
\begin{proof}
We may assume $\la \ge \cf(\ka)$.
Fix a $\ka$-indecomposable uniform ultrafilter $\calU$ over $\ka^+$,
which is $\cf(\ka)$-indecomposable by Theorem \ref{27}.
$\calU$ is $(\ka, \ka^+)$-regular by Theorem \ref{25},
so we can take 
a family $\{x_\alpha \mid \alpha<\ka^+\}$
such that $\size{x_\alpha}<\ka$ and $\{\alpha \mid i\in x_\alpha \} \in \calU$ for $i<\ka^+$.

Fix an enumeration $\{B_i \mid i<\ka^+\}$ of $\bigcup_{\mu<\ka} [\mu]^{\le \la}$
such that for every $j<\ka^+$, the set $\{i \mid B_i=B_j \}$ is cofinal in $\ka^+$.
For $\alpha<\ka^+$,
let $\calB_\alpha=\{\bigcup_{i \in c} B_i \mid c \in [x_\alpha]^{\cf(\ka)} \}$.
Since $\mu^\la \le \ka^+$ for every $\mu<\ka$ and $\size{x_\alpha}<\ka$, we have that 
$\size{\calB_\alpha} \le \ka^+$.
Thus $\size{\bigcup_\alpha \calB_\alpha} \le \ka^+$.
We prove that $[\ka]^\la \subseteq \bigcup_\alpha \calB_\alpha$, 
which complete the proof.

Fix an increasing sequence $\seq{\ka_\xi \mid \xi<\cf(\ka)}$ with limit $\ka$.
For $X \in [\ka]^{\la}$,
we can take an increasing sequence $\seq{i_\xi \mid \xi<\cf(\ka)}$ such that
$X \cap \ka_\xi=B_{i_\xi}$.
Put $d=\{i_\xi \mid \xi<\cf(\ka)\}$, which has order type $\cf(\ka)$.
Because $\{\alpha \mid i_\xi \in x_\alpha\} \in \calU$ for $\xi<\cf(\ka)$ and 
$\calU$ is $\cf(\ka)$-indecomposable, there must be $\alpha<\ka^+$
such that $\size{d \cap x_\alpha}=\cf(\ka)$.
Then $X=\bigcup_{i \in d \cap x_\alpha} B_i \in \calB_\alpha$.
\end{proof}

\begin{thm}\label{7.9}
If $\ka$ is a strong limit singular cardinal,
then $\mathfrak{u}(\ka) \le \mathfrak{u}(\ka^+)$.
\end{thm}
\begin{proof}
If $\mathfrak{u}(\ka^+)<\mathfrak{u}(\ka)$,
then $\ka^+$ has a $\ka$-indecomposable uniform ultrafilter by Corollary \ref{4.2},
hence $2^\ka=\ka^{\cf(\ka)}=\ka^+$
by Proposition \ref{4.15}.
But then $\mathfrak{u}(\ka) \le 2^\ka=\ka^+ \le \mathfrak{u}(\ka^+)$,
this is a contradiction.
\end{proof}
The strong limit assumption
of this theorem cannot be eliminated by Theorem \ref{3.11}.

%\begin{cor}
%Let $\ka$ be a singular cardinal.
%If $\cf(\ka)=\om$ or $\ka$ is strong limit,
%then $\mathfrak{u}(\ka) \le \mathfrak{u}(\ka^{+n})$
%for every $n<\om$.
%\end{cor}

Next we prove that
if $\mathfrak{u}(\ka^+)<\mathfrak{u}(\ka)$,
then $\mathfrak{u}(\ka^+)$ is an upper bound of the
ultrafilter numbers at almost all regular cardinals below $\ka$.
For this sake, we prove some results about (in)decomposable ultrafilters.
% because
%every $\cf(\ka)$-indecomposable ultrafilter is $\ka$-indecomposable. 

Recall that, for a cardinal $\ka$ and an ultrafilter $\calU$,
$\calU$ is \emph{$\ka$-descendingly incomplete}
if there are sets $X_\alpha \in \calU$ ($\alpha<\ka$) such that
for every $\alpha<\beta<\ka$, $X_\alpha \supseteq X_\beta$ holds,
and $\bigcap_{\alpha<\ka} X_\alpha=\emptyset$.
It is easy to see that every $\ka$-decomposable ultrafilter is $\ka$-descendingly incomplete,
and the converse also holds if $\ka$ is regular.

\begin{thm}[Kunen and Prikry \cite{KP}]\label{KPth}
Let $\calU$ be an ultrafilter and $\ka$  a singular cardinal.
If $\calU$ is $\ka^+$-descendingly incomplete but not 
$\cf(\ka)$-descendingly incomplete,
then there is a cardinal $\la<\ka$ such that
$\calU$ is $\mu$-descendingly incomplete for every regular $\mu$ with
$\la<\mu<\ka$.
\end{thm}

This result lead us  to the following notion.
\begin{define}
Let $\ka$ be a limit cardinal and
$\calU$ an ultrafilter.
$\calU$ is  \emph{almost $\mathop{<}\ka$-decomposable}
if 
there is a cardinal $\la<\ka$ such that
$\calU$ is $\mu$-decomposable for every regular $\mu$ with
$\la<\mu<\ka$.
\end{define}
In terms of (in)decomposability, 
Kunen and Prikry's theorem can be restated as:
If $\calU$ is $\ka^+$-decomposable
but $\cf(\ka)$-indecomposable,
then $\calU$ is almost $\mathop{<}\ka$-decomposable.
By Theorem \ref{27}, we can improve Kunen and Prikry's theorem as follows.
Note that every $\cf(\ka)$-indecomposable ultrafilter is $\ka$-indecomposable,
and the converse does not hold in general.

%
%
%
%
%
%
%
%
%
%\end{proof}
%
%
%
%
%
%
%
%
%\end{proof}
%
\begin{prop}\label{37}
Let $\ka$ be a singular cardinal.
If $\calU$ is a $\ka^+$-decomposable 
but $\ka$-indecomposable ultrafilter,
then $\calU$ is almost $\mathop{<}\ka$-decomposable.
In particular,
if $\calU$ is $\ka$-indecomposable uniform ultrafilter over $\ka^+$,
then $\calU$ is almost $\mathop{<}\ka$-decomposable.
\end{prop}
\begin{proof}
By Theorem \ref{27}, we have that $\calU$ is $\cf(\ka)$-indecomposable.
Since $\cf(\ka)$ is regular, $\calU$ is not $\cf(\ka)$-descendingly incomplete.
$\calU$ is $\ka^+$-decomposable, hence is $\ka^+$-descendingly complete.
Then, by Theorem \ref{KPth}, there is a cardinal $\la<\ka$
such that for every regular $\mu$ with $\la<\mu<\ka$,
$\calU$ is $\mu$-descendingly incomplete, in particular $\mu$-decomposable.
\end{proof}

Concerning Kunen and Prikry's theorem, Lipparini showed the following,
which can be seen as a converse of Kunen and Prikry's theorem.
\begin{thm}[Lipparini, Theorem 1 and Corollary 7 in \cite{L}]\label{Li}
Let $\ka$ be a singular cardinal and $\calU$ an ultrafilter.
\begin{enumerate}
\item If $\calU$ is almost $\mathop{<}\ka$-decomposable,
then $\calU$ is $\ka$-decomposable or $\ka^+$-decomposable.
\item If $\calU$ is $\cf(\ka)$-indecomposable,
then $\calU$ is $\ka^+$-decomposable
 if and only if $\calU$ is almost $\mathop{<}\ka$-decomposable.
\end{enumerate}
\end{thm}
We have the following variant of this theorem,
which summarizes Theorem \ref{27} and Proposition \ref{37}.
% by the same argument before.
\begin{cor}\label{7.13}
Let $\ka$ be a singular cardinal and $\calU$ an ultrafilter.
Then the following are equivalent:
\begin{enumerate}
\item $\calU$ is $\ka$-indecomposable but almost $\mathop{<}\ka$-decomposable.
\item $\calU$ is $\cf(\ka)$-indecomposable but almost $\mathop{<}\ka$-decomposable.
\item $\calU$ is $\ka$-indecomposable but $\ka^+$-decomposable.
\item $\calU$ is $\cf(\ka)$-indecomposable but $\ka^+$-decomposable.
\end{enumerate}
\end{cor}
\begin{proof}
(4) $\Rightarrow$ (3) is trivial since every $\cf(\ka)$-indecomposable ultrafilter is
$\ka$-indecomposable.

(3) $\Rightarrow$ (2). Suppose (3). Then $\calU$ is $\cf(\ka)$-indecomposable by Theorem \ref{27},
and is almost $\mathop{<}\ka$-decomposable by Proposition \ref{37}.

(2) $\Rightarrow$ (1) is trivial as the direction (4) $\Rightarrow$ (3).

(1) $\Rightarrow$ (4). Suppose (1). By (1) of Theorem \ref{Li}, $\calU$ is $\ka^+$-decomposable,
and so $\calU$ is $\cf(\ka)$-indecomposable by Theorem \ref{27}.
\end{proof}

Now the following theorem is immediate from Lemma \ref{4.1}, Corollary \ref{4.2} and Proposition \ref{37}.
\begin{thm}\label{5.15}
Let $\ka$ be a singular cardinal,
and suppose $\mathfrak{u}(\ka^+)<\mathfrak{u}(\ka)$.
Then there is a cardinal $\la<\ka$ such that
$\mathfrak{u}(\mu) \le \mathfrak{u}(\ka^+)$ for every regular
$\mu$ with $\la<\mu<\ka$.
\end{thm}

Note that we can improve  ``every regular
$\mu$'' in this theorem
to ``every regular or singular with countable cofinality $\mu$''
by Corollary \ref{29}, however cannot to ``every cardinal $\mu$'';
Suppose $\ka$ is strongly compact, and let $\mu> \ka^{+\om_1^2}$ be a strong limit singular cardinal
with cofinality $\om_1$.
Then $\Add(\om, \mu)$ forces
$\mathfrak{u}(\ka^{+\om_1^2 +1})<\mathfrak{u}(\ka^{+\om_1^2})$
and  $\mathfrak{u}(\ka^{+\om_1^2 +1})<\mathfrak{u}(\ka^{+\om_1\cdot (\alpha+1)})$
for every $\alpha<\om_1$.

We end this section by showing the consistency of
$\mathfrak{u}(\ka^{+\om_1}) <\mathfrak{u}(\om_\om)$.

\begin{lemma}\label{7.18}
Let $\ka$ be a cardinal.
Then $\mathfrak{u}(\ka) \le \mathfrak{u}(\ka^+)^{\cf(\ka)}$.
\end{lemma}
\begin{proof}
It is trivial if $\ka$ is regular, so suppose $\ka$ is singular.
We may assume $\mathfrak{u}(\ka^+)<\mathfrak{u}(\ka)$.
Fix an increasing sequence $\seq{\ka_i \mid i<\cf(\ka)}$ 
of regular cardinals with limit $\ka$.
By Theorem \ref{5.15}, there is $\la<\ka$ such that 
$\mathfrak{u}(\mu) \le \mathfrak{u}(\ka^+)$
for every regular $\mu$ with $\la<\mu<\ka$.
Hence we may assume that 
for every $i<\cf(\ka)$,
there is a uniform ultrafilter $\calU_i$ over $\ka_i$
with $\chi(\calU_i) \le \mathfrak{u}(\ka^+)$.
Fix also a uniform ultrafilter $\calV$ over $\cf(\ka)$.
Now define a filter $\mathcal{D}$ over $\ka$ as:
$X \in \mathcal{D} \iff \{i<\cf(\ka) \mid X \cap \ka_i \in \calU_i\} \in \calV$.
It is easy to show that $\mathcal{D}$ is a uniform ultrafilter over $\ka$,
and since $\chi(\calU_i) \le \mathfrak{u}(\ka^+)$,
we have $\chi(\mathcal{D}) \le \mathfrak{u}(\ka^+)^{\cf(\ka)}\cdot 2^{\cf(\ka)}$.
Thus $\mathfrak{u}(\ka) \le \chi(\mathcal{D}) \le \mathfrak{u}(\ka^+)^{\cf(\ka)}$.
\end{proof}

\begin{prop}\label{7.19}
Suppose there is no $\om_3$-indecomposable uniform ultrafilter over $\om_\om$
and $2^{\om_1}=\om_2$.
Let $\ka$ be a strongly compact cardinal, and 
$\mu>\ka$ a strong limit singular cardinal with $\cf(\mu)=\om_3$.
Then $\Add(\om_2, \mu)$ forces $\mathfrak{u}(\ka^{+\om_1}) < \mathfrak{u}(\om_\om)$.
\end{prop}
\begin{proof}

In the generic extension, 
we have that $\mathfrak{u}(\ka^{+\om_1+1}) <\mathfrak{u}(\om_3)$.
Since $\Add(\om_2, \mu)$ satisfies the $\om_3$-c.c.,
we have $\mathfrak{u}(\om_3) \le \mathfrak{u}(\om_\om)$  as in the proof of
Theorem \ref{7.5},
so we have that $\mathfrak{u}(\ka^{+\om_1+1}) <\mathfrak{u}(\om_\om)$.

We prove that $\mathfrak{u}(\ka^{+\om_1})\le \mathfrak{u}(\ka^{+\om_1+1})$ in the generic extension.
%Suppose to the contrary that $\mathfrak{u}(\ka^{+\om_1+1}) <\mathfrak{u}(\ka^{+\om_1})$.
We know that $\mathfrak{u}(\ka^{+\om_1+1}) \ge \mu$ by Lemma \ref{10}.
By Lemma \ref{7.18},
we have $\mathfrak{u}(\ka^{+\om_1}) \le \mathfrak{u}(\ka^{+\om_1+1})^{\om_1}$.
However, since $\mathfrak{u}(\ka^{+\om_1+1})= \mu$,
$\cf(\mu)=\om_3$, and $\Add(\om_2, \mu)$ is $\om_2$-closed,
we have 
$\mathfrak{u}(\ka^{+\om_1}) \le  \mu^{\om_1}=\mu=$
$\mathfrak{u}(\ka^{+\om_1+1})$.
\end{proof}

\section{Consistency strength of the failure of monotonicity}

In this section, we study the consistency strength of
the failure of monotonicity of the  ultrafilter number function.
We will show that the failure has  large cardinal strength.

We  use the following Donder's theorem.
\begin{thm}[Theorem 4.1 and 4.5 in Donder \cite{Donder}]\label{43}
Suppose there is no inner model with a measurable cardinal.
Let $\ka$ be an uncountable cardinal.
\begin{enumerate}
\item If $\ka$ is singular, then every uniform ultrafilter over $\ka$ is $(\om, \ka)$-regular.
\item If $\ka$ is  regular,  then every uniform ultrafilter over $\ka$ is $(\om, \la)$-regular
for every $\la<\ka$.
\end{enumerate}
\end{thm}
\begin{prop}\label{6.6}
Suppose there is a  cardinal $\ka$ such that
$\ka$ has a $\la$-indecomposable uniform ultrafilter for some $\la<\ka$.
Then there is an inner model with a measurable cardinal.
\end{prop}
\begin{proof}
Let $\calU$ be a $\la$-indecomposable uniform ultrafilter over $\ka$.
Then $\calU$ is not $(\om, \la)$-regular by Lemma \ref{reg-indecomp}.
If $\ka$ is regular, 
by (2) of Theorem \ref{43}, there is an inner model with a measurable cardinal.
Suppose $\ka$ is singular. Since $\la<\ka$ and $U$ is not $(\om, \la)$-regular,
it is easy to check that $U$ is not $(\om, \ka)$-regular.
thus there is an inner model with a measurable cardinal by (1) of Theorem \ref{43}.
%By Proposition \ref{27}, 
%$\calU$ is $\cf(\la)$-indecomposable or $\la^+$-indecomposable.
%Note that $\la^+<\ka$ if $\calU$ is $\la^+$-indecomposable.
%Thus we may assume that $\la$ is regular.
%We prove that $\calU$  is not $(\om, \la)$-regular.
%
\end{proof}

\begin{thm}\label{6.7}
Suppose there are cardinals $\la<\ka$ such that $\mathfrak{u}(\ka)<\mathfrak{u}(\la)$.
Then there is an inner model with a measurable cardinal.
\end{thm}
\begin{proof}
By the assumption Corollary \ref{4.2}, $\ka$ has a $\la$-indecomposable uniform ultrafilter.
Then there is an inner model with a measurable cardinal by Proposition \ref{6.6}.
\end{proof}

By Theorems \ref{4.4}, \ref{4.5}, and \ref{6.7},
we have the equiconsistency result.

\begin{thm}\label{8.4}
The following are equiconsistent:
\begin{enumerate}
\item There is a measurable cardinal.
\item There are cardinals $\la<\ka$
with $\mathfrak{u}(\ka)<\mathfrak{u}(\la)$.
\item There is a weakly inaccessible cardinal $\ka$ and a cardinal $\la<\ka$
with $\mathfrak{u}(\ka)<\mathfrak{u}(\la)$.
\item There is a singular cardinal $\ka$ and a cardinal $\la<\ka$
with $\mathfrak{u}(\ka)<\mathfrak{u}(\la)$.
\item There are cardinals $\la<\ka$ such that
$\ka$ has a $\la$-indecomposable uniform ultrafilter.
\end{enumerate}
\end{thm}

%\begin{cor}
%Let $\ka$ be a cardinal,
%and suppose $\mathfrak{u}(\ka^+)<\mathfrak{u}(\la)$ holds for some $\la \le \ka$.
%Then $\square_\ka$ fails.
%\end{cor}

Here let us mention another application of
Donder's theorem, which contrasts with
Theorem \ref{5.15}.
\begin{prop}
Relative to the existence of a measurable cardinal,
it is consistent that 
there is a weakly inaccessible cardinal $\ka$ such that
 $\mathfrak u(\ka)<\mathfrak{u}(\la)$ for every uncountable cardinal $\la<\ka$.
\end{prop}
\begin{proof}
Suppose there is  a measurable cardinal $\ka$ and a normal measure $U$ over $\ka$
such that $V=L[U]$, and  $V_\ka$ is a model with
$\ZFC+$``there is no inner model with a measurable cardinal''.
By Theorem \ref{43} applied in $V_\ka$, for every uncountable cardinal $\la<\ka$, $\la$ has no $\om_1$-indecomposable uniform ultrafilter.

Fix a large strong limit singular cardinal $\mu>\ka$ with cofinality $\om_1$,
and take a generic extension $V[G]$ via poset $\Add(\om, \mu)$.
In $V[G]$, we have $\mathfrak{u}(\ka) \le \mu <\mathfrak u(\om_1)$.
It is sufficient to show that $\frak{u}(\om_1) \le \frak{u}(\la)$.
Suppose to the contrary that
$\frak{u}(\om_1)> \frak{u}(\la)$.
Then  $\la$ has an $\om_1$-indecomposable uniform ultrafilter by Corollary \ref{4.2}.
Since $V[G]$ is a c.c.c. forcing extension of $V$,
by the argument used in the proof of Theorem \ref{7.5},
we can take an $\om_1$-indecomposable uniform ultrafilter over $\la$ in $V$.
This contradicts to the choice of $V$.
\end{proof}

Next we turn to the consistency strength of the statement
$\mathfrak{u}(\la)<\mathfrak{u}(\ka^+)$ for some $\la \le  \ka$.
For this sake, we will consider the square principles.
\begin{define}
Let $\ka$ be a cardinal.
A sequence $\seq{c_\alpha \mid \alpha<\ka^+}$
is said to be a \emph{$\square_\ka$-sequence}
if the following hold:
\begin{enumerate}
\item $c_\alpha$ is a club in $\alpha$ with order type $\le \ka$.
\item For every $\alpha<\beta<\ka^+$,
if $\alpha$ is a limit point of $c_\beta$ then $c_\alpha=c_\beta \cap \alpha$.
\end{enumerate}
Let $\square_\ka$ be the assertion that 
there is a $\square_\ka$-sequence.
%\item A sequence $\seq{C_\alpha \mid \alpha<\ka^+}$
%is called a \emph{$\square_\ka^*$-sequence}
%if the following hold:
%\begin{enumerate}
%\item $C_\alpha$ is a set of clubs in $\alpha$ with order type $\le \ka$.
%\item $\size{C_\alpha} \le \ka$.
%\item For every $\alpha<\beta<\ka^+$ and $c \in C_\beta$,
%if $\alpha$ is a limit point of $c$ then $c \cap \alpha \in C_\alpha$.
%\end{enumerate}
%Let $\square_\ka^*$ be the assertion that 
%there is a $\square_\ka$-sequence.
%\end{enumerate}
\end{define}

\begin{define}
Let $\ka$ be a  cardinal.
A sequence $\seq{c_\alpha \mid \alpha<\ka}$
is said to be a \emph{$\square(\ka)$-sequence}
if the following hold:
\begin{enumerate}
\item $c_\alpha$ is a club in $\alpha$.
\item There is no club $C$ in $\ka$ such that
for every $\alpha<\ka$, if $\alpha$ is a limit point of $C$
then $C \cap \alpha=c_\alpha$.
\end{enumerate}
Let $\square(\ka)$ be the assertion that 
there is a $\square(\ka)$-sequence.
\end{define}
Clearly $\square_\ka \Rightarrow \square(\ka^+)$.
It is known that the failure of the square principle at a singular cardinal,
or at successive cardinals have very strong large cardinal strength,
but the exact strength is unknown.
For example, the following theorem gives such a lower bound.
\begin{thm}[Schimmerling \cite{Schi}]\label{fail square}
Suppose that:
\begin{enumerate}
\item $\square_\ka$ fails for some singular $\ka$, or
\item Both $\square(\ka)$ and $\square_\ka$ fail for some $\ka \ge \om_2$.
\end{enumerate}
Then there is an inner model with a proper class of strong cardinals.
\end{thm}

The following proposition was already proved by
Lambie-Hanson and Rinot (\cite{LHR}), and Inamdar and Rinot (\cite{IR}),
in fact they obtained more general and strong results.
However we give a proof of it  for the completeness.

\begin{prop}\label{48}
Let $\ka$ be a regular uncountable cardinal,
and suppose there is a regular $\la <\ka$ such that 
$\ka$ has a $\la$-indecomposable uniform ultrafilter $\calU$.
Then $\square(\ka)$ fails.
\end{prop}

For this sake, we need the next definition and yet another result of Kanamori.

\begin{define}
Let $\calU$ be an ultrafilter over $\ka$.
\begin{enumerate}
\item A function $f:\ka \to \ka$
is a \emph{least function} of $\calU$
if $\{\alpha<\ka \mid \beta<f(\alpha)\} \in \calU$ for every $\beta<\ka$,
and for every $g:\ka\to \ka$,
if $\{\alpha \mid g(\alpha)<f(\alpha)\} \in \calU$,
then there is $\beta<\ka$ such that
$\{\alpha \mid g(\alpha) \le \beta\} \in \calU$.
\item If the identity map is a least function,
then $\calU$ is called a \emph{weakly normal ultrafilter}.
\end{enumerate}
\end{define}
Note that if $\calU$ is an ultrafilter over $\ka$ and $f:\ka \to \ka$,
then the ultrafilter $f_*(\calU)$ over $\ka$ 
is weakly normal if and only if $f$ is a least function of $\calU$.

\begin{thm}[Kanamori, Corollary 2.6 in \cite{Kanamori2}]\label{47}
Let $\ka$ be a regular cardinal,
and $\calU$ a uniform ultrafilter over $\ka$.
If there is $\la<\ka$ such that
$\calU$ is not $(\om, \la)$-regular,
then $\calU$ has a least function.
\end{thm}

By Lemma \ref{reg-indecomp},
if $\calU$ is a uniform ultrafilter over $\ka$
and there is a $\la<\ka$ such that
$\calU$ is $\la$-indecomposable,
then $\calU$ is not $(\om, \la)$-regular.
Hence $\calU$ has a least function.

Now we start the proof of Proposition \ref{48}.
\begin{proof}[Proof of Proposition \ref{48}]
To show that $\square(\ka)$ fails,
suppose to the contrary that there is a $\square(\ka)$-sequence $\seq{c_\alpha \mid \alpha<\ka}$.
By Theorem \ref{47}, $\calU$ has a least function.
Hence we may assume that $\calU$ is weakly normal.

\begin{claim}
For every $\beta<\ka$,
there is $\gamma<\ka$ such that $\beta<\gamma$
and $\{\alpha<\ka \mid \gamma \in \lim(c_\alpha)\} \in \calU$.
\end{claim}
\begin{proof}
We shall take a strictly increasing sequence $\seq{\beta_i \mid i <\la}$ by induction on
$\la$ such that for every $i<\la$,
$\{\alpha<\ka \mid [\beta_i, \beta_{i+1}) \cap c_\alpha \neq\emptyset\} \in \calU$.
First let $\beta_0=\beta$.
Suppose $\beta_j$ is defined for all $j<i$.
If $i$ is limit, then $\beta_i=\sup_{j<i} \beta_j$.
Suppose $i$ is successor, say $i=k+1$.
By the weak normality of $\calU$,
we can find $\beta_{k+1}$
such that
$\{\alpha \mid [\beta_k, \beta_{k+1}) \cap c_\alpha \neq \emptyset\} \in \calU$.

Let $\gamma=\sup_i \beta_i$. Clearly $\gamma>\beta$ and $\cf(\gamma)=\la$.
We show that 
$\{\alpha<\ka \mid \gamma \in \lim(c_\alpha)\} \in \calU$.
For $i<\la$, let $X_i=\{\alpha \mid [\beta_i, \beta_{i+1}) \cap c_\alpha \neq \emptyset\} \in \calU$.
Since $\calU$ is $\la$-indecomposable,
the set $Y=\{\alpha \mid \{i<\la \mid \alpha \in X_i\}$ is cofinal in $\la\}$ is in $\calU$.
Take $\alpha \in Y$.
By the choice of the $X_i$'s, we have that $c_\alpha \cap \gamma$ is unbounded in $\gamma$,
hence $\gamma \in \lim(c_\alpha)$.
\end{proof}

Let $A$ be the set of all $\gamma<\ka$ such that
$\{\alpha<\ka \mid \gamma \in \lim(c_\alpha)\} \in \calU$.
$A$ is unbounded in $\ka$ by the previous claim.
For $\beta_0, \beta_1 \in A$ with $\beta_0<\beta_1$,
there is a large $\alpha$ such that $\beta_0, \beta_1 \in \lim(c_\alpha)$.
Hence $c_{\beta_1} \cap \beta_0=c_{\beta_0}$.
This argument shows that
$C=\bigcup\{c_\beta \mid \beta \in A\}$ is a club in $\ka$ and
$C \cap \beta=c_\beta$ for every $\beta \in \lim(C)$.
This is a contradiction.
\end{proof}

\begin{cor}\label{8.13}
\begin{enumerate}
\item Let $\ka$ be a regular cardinal.
If $\mathfrak{u}(\ka)<\mathfrak{u}(\la)$ holds for some $\la<\ka$,
then $\square(\ka)$ fails.
\item If $\mathfrak{u}(\ka^+)<\mathfrak{u}(\la)$ holds for some $\la \le \ka$,
then $\square_\ka$ fails.
\end{enumerate}
\end{cor}
\begin{proof}
(2) is immediate from (1).

(1). By the assumption, we can find  a $\la$-indecomposable uniform ultrafilter $\calU$.
If $\la$ is regular, then $\square(\ka)$ fails by Proposition \ref{48}.
If $\la$ is singular,
then $\calU$ is either $\cf(\la)$-indecomposable, or $\la^+$-indecomposable by Theorem \ref{27}.
In either case, we have that $\square(\ka)$ fails by Proposition \ref{48} again.
\end{proof}

\begin{prop}
Let $\ka$ be a regular cardinal,
and suppose there is a cardinal $\la \le \ka$ such that 
$\mathfrak{u}(\ka^+)<\mathfrak{u}(\la)$ holds.
Then both $\square(\ka)$ and $\square(\ka^+)$ fail.
\end{prop}
\begin{proof}
First note that in fact $\la \neq \ka$ by Proposition \ref{21}.
By the assumption, $\ka^+$ has a $\la$-indecomposable uniform ultrafilter $\calU$
by Corollary \ref{4.2}.
Now we show that there is a regular cardinal $\mu < \ka$
such that $\calU$ is $\mu$-indecomposable.
If $\la$ is regular, just let $\mu=\la$.
Suppose $\la$ is singular.
By Theorem \ref{27}, $\calU$ is either $\cf(\la)$-indecomposable,
or else $\la^+$-indecomposable.
If it is $\cf(\la)$-indecomposable, then let $\mu=\cf(\la)$.
If it is $\la^+$-indecomposable, we have $\la^+<\ka$
by Lemma \ref{2.4}. 
Put $\mu=\la^+$ in this case.

We have that $\square(\ka^+)$ fails by Proposition \ref{48}.
$\calU$ is $\ka$-decomposable by 
Lemma \ref{2.4}, so we can find a $\mu$-indecomposable uniform ultrafilter over $\ka$
by Lemma \ref{2.8}.
By applying Proposition \ref{48} again,
$\square(\ka)$ fails.
\end{proof}

\begin{prop}
Let $\ka$ be a singular cardinal.
If  $\mathfrak{u}(\ka^+)<\mathfrak{u}(\ka)$,
then the set $\{\la<\ka \mid \la$ is regular and $\square(\la)$ holds $\}$ is bounded in $\ka$.
\end{prop}
\begin{proof}
By the assumption, we can take a $\ka$-indecomposable uniform ultrafilter $\calU$ over $\ka^+$,
which is in fact $\cf(\ka)$-indecomposable by Theorem \ref{27}.
Now suppose to contrary  that the set $\{\la<\ka \mid \la$ is regular and 
$\square(\la)$ holds $\}$ is unbounded in $\ka$.
By Proposition \ref{37},
we can take a regular $\la<\ka$ such that $\la>\cf(\ka)$,
$\calU$ is $\la$-decomposable, and $\square(\la)$ holds.
Then we can take a $\cf(\ka)$-indecomposable uniform ultrafilter $\calV$ over $\la$
by Lemma \ref{2.8}.
Hence $\square(\la)$ fails by Proposition \ref{48}, this is a contradiction.
\end{proof}

By Theorem \ref{fail square} and these propositions,
we obtain the following lower bound of the consistency strength.
\begin{thm}\label{8.16}
Let $\ka$ be a cardinal.
If $\mathfrak{u}(\ka^+)<\mathfrak{u}(\la)$ holds for some cardinal $\la \le \ka$,
then there is an inner model with  a proper class of strong cardinals.
\end{thm}

To conclude this paper, we would like to ask some questions.
\begin{question}
\begin{enumerate}
%\item Is it possible that $\cf(\mathfrak u(\ka))=\cf(\ka)$?
%\item What is the exact consistency strength of
%the statement that ``$\mathfrak{u}(\ka^+)<\mathfrak{u}(\ka)$ for some singular $\ka$''?
\item What is the exact consistency strength of
the statement ``$\mathfrak{u}(\ka^+)<\mathfrak{u}(\la)$ for some $\la \le \ka$''?
%\item Is it consistent that there are singular cardinals $\la<\ka$ such that
%$\cf(\la)=\cf(\ka)>\om$ and $\mathfrak{u}(\ka)<\mathfrak{u}(\la)$?
%\item Is it consistent that $\mathfrak u (\ka)<\mathfrak u (\la)$ for some $\la<\ka$ with $\cf(\la)=\om$?
%\item Is it consistent that
%$\mathfrak{u}(\ka)^+<\mathfrak{u}(\la)$ for some $\la<\ka$?
\item Is it consistent that there are cardinals $\la<\ka$ such that
$2^{<\la}=\la$ and $\mathfrak{u}(\ka)<\mathfrak{u}(\la)$?
%\item Is it consistent that $\mathfrak{u}(\ka)<\mathfrak{u}(\la)$ and $\mathfrak{u}(\ka)$ is regular 
%for some $\la<\ka$?
\item Is it consistent that monotonicity of the ultrafilter number function fails at three cardinals?
Namely, is it consistent that 
there are three cardinals $\ka_0<\ka_1<\ka_2$ with $\mathfrak{u}(\ka_2)<\mathfrak{u}
(\ka_1)<\mathfrak{u}(\ka_0)$? How about
$\mathfrak{u}(\ka_1)<\mathfrak{u}(\ka_2)<\mathfrak{u}(\ka_0)$ and
$\mathfrak{u}(\ka_2)<\mathfrak{u}
(\ka_0)<\mathfrak{u}(\ka_1)$?
\item Is there anther cardinal invariant function $\mathfrak{k}(\ka)$ on the cardinals such that
the failure of monotonicity of $\mathfrak{k}$ has large cardinal strength?
\end{enumerate}
\end{question}

\subsection*{Acknowledgments}
The author  deeply grateful to the referee for many insightful suggestions and comments.
The author also would like to thank Assaf Rinot for many useful comments.

\printindex

\end{document}